\documentclass{amsart}
\usepackage{amsmath,amsthm, amssymb,xypic, graphicx, mathabx}
\usepackage{longtable, caption}
\usepackage{cite,  mathscinet}
\usepackage{todonotes}
\usepackage[colorlinks]{hyperref}
\usepackage{multicol}

\usepackage{cleveref}

\newcommand{\zerocycle}[2]{X^{\vec{#1}}_{#2}}

\newcommand{\dlog}{d\log}
\newcommand{\calR}{\mathcal{R}}
\newcommand{\calH}{\mathcal{H}}

\newcommand{\calZ}{\mathcal{Z}}
\newcommand{\Kap}{\mathrm{Kap}}
\newcommand{\ul}[1]{\vec{#1}}
\newcommand{\LL}{\bbL}
\newcommand{\bt}{\mathbf{t}}

\numberwithin{equation}{subsubsection}
\theoremstyle{plain}

\newtheorem{maintheorem}{Theorem}

\newtheorem*{theorem*}{Theorem}
\newtheorem*{metaconjecture}{Meta-conjecture}

\newtheorem{theorem}[subsubsection]{Theorem}
\newtheorem{corollary}[subsubsection]{Corollary}
\newtheorem{conjecture}[subsubsection]{Conjecture}
\newtheorem{proposition}[subsubsection]{Proposition}
\newtheorem{lemma}[subsubsection]{Lemma}

\theoremstyle{definition}

\newtheorem{example}[subsubsection]{Example}
\newtheorem{definition}[subsubsection]{Definition}
\newtheorem{remark}[subsubsection]{Remark}

\theoremstyle{remark}

\newcommand{\Gal}{\mathrm{Gal}}

\newcommand{\Conf}{C}
\newcommand{\Sym}{\mathrm{Sym}}

\newcommand{\Var}{\mathrm{Var}}

\newcommand{\HS}{\mathrm{HS}}

\newcommand{\calL}{\mathcal{L}}
\newcommand{\calO}{\mathcal{O}}
\newcommand{\calM}{\mathcal{M}}

\newcommand{\bbN}{\mathbb{N}}
\newcommand{\bbZ}{\mathbb{Z}}
\newcommand{\bbQ}{\mathbb{Q}}
\newcommand{\bbR}{\mathbb{R}}
\newcommand{\bbC}{\mathbb{C}}
\newcommand{\bbP}{\mathbb{P}}
\newcommand{\bbA}{\mathbb{A}}
\newcommand{\bbL}{\mathbb{L}}
\newcommand{\bbF}{\mathbb{F}}

\newcommand{\bc}{\mathbf{c}}

\newcommand{\Spec}{\mathrm{Spec}\,}

\begin{document}

\title{Zeta statistics and Hadamard functions}
\author{Margaret Bilu}
\address{IST Austria \\
Am Campus 1 \\
3400 Klosterneuburg, Austria}

\author{Ronno Das}
\address{Department of Mathematics\\
University of Chicago\\
5734 S University Ave\\
Chicago, IL 60637, USA}

\author{Sean Howe}
\address{Department of Mathematics \\
155 South 1400 East, JWB 233 \\
Salt Lake City\\ UT~84112, USA}

\begin{abstract} We introduce the Hadamard topology on the Witt ring of rational functions, giving a simultaneous refinement of the weight and point-counting topologies. Zeta functions of algebraic varieties over finite fields are elements of the rational Witt ring, and the Hadamard topology allows for a conjectural unification of results in arithmetic and motivic statistics: The completion of the Witt ring for the Hadamard topology can be identified with a space of meromorphic functions which we call Hadamard functions, and we make the meta-conjecture that any ``natural" sequence of zeta functions which converges to a Hadamard function in both the weight and point-counting topologies converges also in the Hadamard topology. For statistics arising from Bertini problems, zero-cycles or the Batyrev-Manin conjecture, this yields an explicit conjectural unification of existing results in motivic and arithmetic statistics that were previously connected only by analogy. As evidence for our conjectures, we show that Hadamard convergence holds for many natural statistics arising from zero-cycles, as well as for the motivic height zeta function associated to the motivic Batyrev-Manin problem for split toric varieties. 
\end{abstract}

\maketitle

\tableofcontents

\section{Introduction}

\subsection{Zeta functions}
\subsubsection{Zeta functions of varieties over finite fields}
The  zeta function of a variety $X/\bbF_q$ is the formal power series $Z_X(t)\in 1+t\bbZ[[t]]$ defined by
\[ Z_X(t) := \prod_{x \in X \textrm{ closed }} \frac{1}{1-t^{\deg x}}  = \sum_{j=0}^\infty |\Sym^j X(\bbF_q)| t^j. \]
It encodes the number of points of $X$ over every finite extension of $\bbF_q$:
\begin{equation}\label{eqn.zeta-point-counts} \dlog Z_X(t) = \sum_{j=1}^\infty |X(\bbF_{q^j})|t^{j-1}. \end{equation}
The Grothendieck-Lefschetz fixed point formula implies that $Z_X(t)$ is the power-series expansion at $0$ of a rational function, and that the zeroes and poles are determined by the eigenvalues of the Frobenius acting on the \'{e}tale cohomology of $X$ (up to cancellation between odd and even degree). We write $\calR_1$ for the set of rational functions $f \in \bbC(t)$ such that $f(0)=1$, and from now on consider $Z_X(t)$ as an element of $\calR_1$.

\subsubsection{Grothendieck rings of varieties}\label{sss.groth-ring} Let $K$ be a field. We write $K_0(\Var/K)$ for the \textbf{modified} Grothendieck ring of varieties over $K$ --- it is the free abelian group on isomorphism classes $[X]$ of (not necessarily connected) varieties $X/K$, modulo the relations $[X_1 \sqcup X_2]=[X_1] + [X_2]$ and $[X]=[Y]$ if there is a map $X \rightarrow Y$ inducing a bijection on points over any algebraically closed field. This definition is equivalent to the classical definition of $K_0(\Var/K)$ via cut and paste relations in characteristic zero, but is better behaved in positive characteristic; we refer to \cite[Section 2]{bilu-howe} for a detailed discussion. Write $\bbL:=[\bbA^1] \in K_0(\Var/K),$ and 
\[ \calM_K:= K_0(\Var/K)[\bbL^{-1}]. \]
 There is a natural topology on $\calM_K$ induced by the \emph{dimensional filtration}: for every $d\in \bbZ$, we define $\mathrm{Fil}_d\calM_K$ to be the subgroup of $\calM_K$ generated by elements of the form $[X]\bbL^{-n}$ where $X$ is a variety over $K$ and $\dim X - n\leq d$. This gives us an increasing and exhaustive filtration on the ring $\calM_K$. Denote the completion for this filtration by $\widehat{\calM}_K$.

For a quasi-projective variety $X$ over $K$, consider the \emph{Kapranov zeta function}
$$Z_X^{\Kap}(t) := \sum_{j=0}^{\infty}[\Sym^j X]t^j \in 1 + t\calM_K[[t]],$$
When the field $K$ is finite, $Z_{X}^{\Kap}(t)$ specializes to $Z_{X}(t)$ via point counting, which replaces the class $[\Sym^j X]$ with the number $|\Sym^j X(K)|$ in the coefficients.

\subsection{Arithmetic and motivic statistics}
Many results in arithmetic statistics can be interpreted in terms of asymptotic properties of the number of $\bbF_q$-points on a sequence of varieties $X_n/ \bbF_q$. Such results often have analogs in \emph{motivic} statistics --- these are asymptotic statements in $\widehat{\calM}_K$. A beautiful example of such a correspondence between arithmetic and motivic results is given by the following:

\begin{theorem*} Let $K$ be a field and let $X\subset \bbP^n_K$ be a smooth projective variety. Denote by $U_d$ the open subset of hypersurface sections in $\Gamma(\bbP^n_K, \mathcal{O}(d))$ which intersect $X$ transversely. Then:
\begin{enumerate} \item (Poonen \cite{poonen:bertini}) Assume $K = \bbF_q$ is finite. Then
$$\lim_{d\to \infty} \frac{|U_d(\bbF_q)|}{|\Gamma(\bbP^n,\mathcal{O}(d))(\bbF_q)|}  = Z_{X}(q^{-\dim X - 1})^{-1}.$$
\item  (Vakil-Wood \cite{vakil-wood,vakil-wood-errata}) In $\widehat{\calM}_K$, we have
$$\lim_{d\to \infty} \frac{[U_d]}{[\Gamma(\bbP^n,\mathcal{O}(d))]} = Z_{X}^{\Kap}(\bbL^{-\dim X - 1})^{-1}.$$
\end{enumerate}
\end{theorem*}


Although the two results are compellingly similar for $K= \bbF_q$, neither of the two implies the other: the assignment sending a variety $X/\bbF_q$ to the number of points $|X(\bbF_q)|$ extends to the \emph{point counting measure}, a map of rings 
$\calM_{\bbF_q} \rightarrow \bbR$, but it is not continuous for the dimensional topology on $\calM_{\bbF_q}$. The aim of this paper is to formulate a conjectural unification of such parallel statements in arithmetic and motivic statistics, and study some aspects of this conjecture.

The fundamental insight behind our conjecture is that, through the lens of the \textit{zeta measure} (see (\ref{eqn-zeta-measure}) below), the aforementioned results of Poonen and Vakil-Wood may be viewed as convergence statements of the zeta functions of the varieties $U_d$, suitably renormalized, for two different, incompatible topologies on $\calR_1$. Indeed, via (\ref{eqn.zeta-point-counts}) we can interpet Poonen's theorem, applied simultaneously over all finite extensions of $\bbF_q$, as a convergence result for the zeta functions $Z_{U_d}(t)$ in the \textit{coefficient topology} on $\calR_1$, induced by the product topology on the coefficients of the power series at zero. Because of this interpretation, we refer to this topology on $\calR_1$ also as the \textit{point-counting topology}.
On the other hand, Vakil and Wood's result, via the Weil conjectures, implies a convergence statement about the functions $Z_{U_d}(t)$ in the \textit{weight topology} on $\calR_1$, where a function is considered small if all of its poles and zeroes are at large complex numbers. The point-counting and weight topologies are incompatible; below we introduce the \emph{Hadamard topology} which refines both. 

\subsection{Rings of zeta functions}
The set $\calR_1$ has a ring structure given by Witt addition and multiplication: if we identify $\calR_1$ with the group ring $\bbZ[\bbC^\times]$ via
\[ f(t) \rightarrow - \mathrm{Div} f(t^{-1}), \textrm{ so, e.g., }  (1-at)^{-1} \mapsto [a]\]
then Witt addition and multiplication are induced by addition and multiplication on the group ring $\bbZ[\bbC^\times]$. Alternatively, Witt addition $+_W$ is regular multiplication of rational functions: $f +_W g = fg$, and Witt multiplication $*_W$ is determined by
\[ \dlog f = \sum a_j t^j \textrm{ and } \dlog g = \sum b_j t^j \implies \dlog (f *_W g) = \sum a_j b_j t^j. \]
The assignment $X \mapsto Z_X(t)$ extends to the \emph{zeta measure}, which is a map of rings
\begin{equation}\label{eqn-zeta-measure}
\begin{array}{rcl}\calM_{\bbF_q}  &\rightarrow &\calR_1 \\
a & \mapsto& Z_a(t).
\end{array}
\end{equation}
\subsubsection{The Hadamard topology}
Under the identification $\calR_1 = \bbZ[\bbC^\times]$, the weight topology is induced by the norm
\[ \left\| \sum a_n [z_n] \right\|_\infty = \sup |z_n|. \]
The point-counting topology is induced by the family of seminorms
\[ \left\|\sum a_n [z_n] \right\|_{j} = \left| \sum a_n z_n^j \right| \textrm{ for all integers $j\geq 1$. } \]
We consider also the \emph{Hadamard} topology, defined by the Hadamard norm
\[ \left\| \sum a_n [z_n] \right\|_H = \sum |a_n||z_n|. \]
The Hadamard topology refines both the weight and point-counting topologies. Moreover, the completion of $\calR_1$ for the Hadamard norm is naturally identified with a genuine space of meromorphic functions (as opposed to the completion for the weight topology, which is a space of formal divisors, or the completion for the point-counting topology, which is a space of formal power series):
\begin{definition} 
A \emph{Hadamard function} is a meromorphic function on $\bbC$ that can be written as a quotient $\frac{f}{g}$ where $f$ and $g$ are  entire functions of genus zero. 
\end{definition}

We write $\calH_1$ for the set of Hadamard functions $f$ such that $f(0)=1$. The Hadamard factorization theorem then yields
\begin{theorem}
The completion of $\calR_1$ for $|| \cdot ||_{H}$ is canonically identified with $\calH_1$.
\end{theorem}

\subsection{The meta-conjecture}

Because the Hadamard topology refines both the point-counting and weight topologies, asymptotics in the Hadamard topology give a common refinement of results in arithmetic and motivic statistics. Moreover, taking limits in the Hadamard topology retains the essential analytic characteristics of zeta functions, because these limits can be interpreted as meromorphic functions. For these reasons, it is natural to try to refine previous results in arithmetic and motivic statistics by studying them in the Hadamard topology. And in fact, we conjecture that any natural asymptotic which holds in both the weight and point-counting topologies should also hold in the Hadamard topology:

\begin{metaconjecture}\label{conj.meta} If $a_n \in \calM_{\bbF_q}$ is a ``natural" sequence of classes such that the sequence of zeta functions $Z_{a_n}(t)$ converges in both the point-counting and weight topology to some $f(t) \in \calH_1$, then $Z_{a_n}(t) \rightarrow f(t)$ also in the Hadamard topology.
\end{metaconjecture}

The condition that $f \in \calH_1$ is essential --- without this condition there is no way to compare limits in the point-counting and weight topologies. Moreover, there are natural examples where limits exist in both topologies, but at least one of these limits is not a Hadamard function (cf. \S\ref{subsub:non-example}).

\subsubsection{Hadamard convergence for Bertini problems}
The theorems of Poonen and Vakil-Wood discussed above furnish an example where our meta-conjecture should apply. To see this, we must verify that the special value of the Kapranov zeta function appearing there is in fact a Hadamard function: We apply the zeta measure coefficientwise to $Z_{X}^{\Kap}(s)$ to obtain a series with coefficients in the ring $\calR_1$, and then evaluate at $s=Z_{\bbL^{-m}}(t)=\frac{1}{1-q^{-m}t}$ for $m = \dim X+1$. Indeed, for any $m> \dim X$, if we write
\[ \zeta^\Kap_X(m) := 1 + Z_X(t) s + Z_{\Sym^2 X}(t) s^2 + \cdots |_{s=Z_{\bbL^{-m}}(t)} = \prod_{j \geq 1} Z_{\Sym^j X}(t q^{-mj}), \]
then the infinite product on the right (an infinite sum in the Witt ring structure) converges in the Hadamard topology to an invertible (for Witt multiplication) element of $\calH_1$. Thus, in this case the meta-conjecture specializes to 

\begin{conjecture}\label{conj.zeta-convergence-general} Let $X \subset \bbP^n_{\bbF_q}$ be a smooth projective variety and let $U_d \subset \Gamma(\bbP^n, \calO(d))$ be the open subvariety of hypersurfaces intersecting $X$ transversely. Then, in the Hadamard topology,
\begin{align*} \lim_{d \rightarrow \infty} Z_{U_d}( q^{-\dim U_d} t)&  =  1/_W \zeta_X^{\Kap}(\dim X + 1).
\end{align*}
\end{conjecture}
\noindent Here the notation $/_W$ denotes division in the Witt ring. 

We state separately the case $X = \bbP^n$, which has a particularly simple form:
\begin{conjecture}\label{conj.zeta-convergence-Pn}
Let $U_d \subset V_d := \Gamma(\bbP^n, \calO(d))$ be the space of smooth hypersurfaces of degree $d$ in $\bbP^n$. Then,
\[ \lim_{d\rightarrow \infty} Z_{U_d/\bbF_q}\left(q^{-\dim V_d} t\right) = Z_{\mathrm{GL}_{n+1}/\bbF_q}\left(q^{-(n+1)^2}t\right) \]
in the Hadamard topology. In particular, the sequence of rational functions
\[ \frac{ Z_{U_d/\bbF_q}(q^{-\dim V_d} t) }{Z_{\mathrm{GL}_{n+1}/\bbF_q}(q^{-(n+1)^2}t)} \]
converges uniformly on compact sets in $\bbC$ to the constant function $1$.
\end{conjecture}

For $n=1$, Conjecture~\ref{conj.zeta-convergence-Pn} is true, because for all $d \geq 2$
\[ \frac{[U_d]}{[V_d]} = [\mathrm{GL}_2]\bbL^{-4} \in \calM_{\bbF_q}. \]
For $n > 1$, however, already Conjecture~\ref{conj.zeta-convergence-Pn} is completely open. As some partial evidence, we note that Tommasi \cite{tommasi} has established a cohomological stabilization result for moduli of smooth hypersurfaces in $\mathbb{P}^n_{\mathbb{C}}$ --- cf. \S\ref{subsub:betti-hadamard} below for more details on the relation between cohomological stabilization and Hadamard convergence. 

\begin{remark} Some of the material on Hadamard convergence developed in this work appeared already in the first version of \cite{bilu-howe} posted on arXiv. In particular, it was claimed there that Conjecture~\ref{conj.zeta-convergence-general} could be proved in the case that $\dim X =1$. No details were provided, and there was a mistake in the envisioned proof.\end{remark}

The point-counting \cite{poonen:bertini} and motivic \cite{bilu-howe} Bertini theorems with Taylor coefficients furnish many more examples where we expect that the meta-conjecture should apply. The limits appearing in these theorems are special values of (motivic) Euler products, but, unfortunately we are currently unable to prove that these special values are Hadamard functions in any level of generality!

\subsection{Results for zero-cycles} 
Our meta-conjecture was originally motivated by the Bertini examples discussed above, but for now these seem to be out of reach. On the other hand, there are a number of questions about zero-cycles that have been previously studied in both arithmetic and motivic statistics for which we can both formulate and prove concrete instances of the meta-conjecture. In particular, building on \cite{farb-wolfson-wood, vakil-wood, chen:0-cycles, farb-wolfson:etale-stability, howe:mrv1}, we treat various problems involving colored effective zero-cycles with prescribed incidence relations. We also give an application to the motivic Batyrev-Manin conjecture as in \cite{bourqui}. These are the main results of this paper, and the main evidence that our meta-conjecture is reasonable. 

\subsubsection{Pattern-avoiding zero-cycles}

In \S\ref{section:zero-cycles} we carry out a general study of convergence for densities of spaces of effective zero-cycles with prescribed allowable sets of labels. These generalize different densities considered previously in related contexts by Bourqui \cite{bourqui}, Farb-Wolfson-Wood \cite{farb-wolfson-wood}, and Vakil-Wood \cite{vakil-wood}. We establish fairly complete weight and point-counting convergence results, and find natural examples (and non-examples) of Hadamard convergence. For more details, we refer the reader to the beginning of \S\ref{section:zero-cycles}; below we only highlight some examples.

\subsubsection{Orthogonal pattern-avoiding zero-cycles.}
\label{subsub.intro-pattern-avoiding}
Let $X/\bbF_q$ be a quasi-projective variety and $k\geq 1$ an integer. For $\ul{d}=(d_1, \ldots, d_k)\in \bbZ^k_{\geq 0}$, we write
\[ \Sym^{\ul{d}}X := \Sym^{d_1} X \times \Sym^{d_2} X \times \cdots \times \Sym^{d_k} X. \]
For $K/\bbF_q$ algebraically closed, we can view a point $s \in \Sym^{\ul{d}}(K)$ as a tuple $(s_1, \ldots, s_k)$ of finitely supported functions on $X(K)$ with values in $\bbZ_{\geq 0}$. In particular, for each $x \in X(K)$, we obtain a label vector $\ell_s(x):=(s_1(x), \ldots, s_k(x)) \in \bbZ_{\geq 0}^k$. If we fix a finite subset $V \subset \bbZ_{\geq 0}^k$, then we can consider the locus
\[ \calZ^{\ul{d}}_V(X) \subset \Sym^{\ul{d}} X \]
whose $K$-points for algebraically closed $K$ are exactly those $s$ such that, for all $\vec{v} \in V$ and $x \in X(K)$, $\ell_s(x) \not \geq \vec{v}$ (i.e., $\ell_s$ avoids all of the patterns in $V$).

\begin{example}\label{example:intro-farb-wolfson-wood}
If $V = \{ (n, n, \ldots, n) \}$, then $\calZ_V^{\ul{d}}(X)$ is the subvariety denoted $\calZ_n^{\ul{d}}(X)$ in \cite{farb-wolfson-wood}, which parameterizes tuples of effective zero cycles whose overlap has multiplicity bounded by $n$.  In particular, $\calZ_{ \{ (2) \} }^{(d)}(X) = \Conf^ d X,$ the configuration space of $d$ unordered distinct points on $X$. 
\end{example}

A set of vectors $V$ is orthogonal (for the standard inner product) if and only if for each $1 \leq i \leq k$, there is at most one vector $\vec{v} \in V$ with non-zero $i$th component. For $\vec{v}=(v_1, \ldots, v_k) \in \bbZ_{\geq 0}^k$, we write $|\vec{v}|=v_1 + \ldots + v_k$. We say that a set of vectors $V$ is non-degenerate if it does not contain a $\vec{v}$ with $|\vec{v}|\leq 1$ (i.e. it does not contain the zero vector or the unit vector $e_i$ for any $i$). 

\begin{maintheorem}\label{theorem:pattern-avoiding-0-cycles} If $V$ is orthogonal and non-degenerate, then, in the Hadamard topology on $\calH_1$,
 \[ \lim_{d_1, d_2, \ldots, d_k \rightarrow \infty} Z_{\calZ_V^{\vec{d}}(X)}(t) /_W Z_{\Sym^{\vec{d}}(X)}(t) = 1 /_W \left( \prod_{\vec{v}\in V} \zeta_X^{\Kap}\left(|\vec{v}|\cdot \dim X\right) \right). \]
Note that all ring operations in this equation are taken in the Witt ring structure.
\end{maintheorem}

\begin{remark}
If $V$ contains $0$ then $\calZ_V^{\ul{d}}(X)=\emptyset$ for any $\vec{d}$, and if $V$ contains $e_i$ then $\calZ_V^{\vec{d}}(X)=\emptyset$ when $d_i \neq 0$. For $V$ not orthogonal, see Section~\ref{subsub:mobius-functions}, particularly Theorem~\ref{theorem:convergence-pattern-avoiding}.
\end{remark}

This theorem, and our other related results, provide a motivic lift of Theorem 1.9-2 of Farb-Wolfson-Wood \cite{farb-wolfson-wood}, which describes the same phenomenon at the level of Hodge-Deligne polynomials in the special case of Example~\ref{example:intro-farb-wolfson-wood}.  This confirms the expectation of a motivic analog stated in \cite[paragraph following Theorem 1.9]{farb-wolfson-wood}. Our proof is based on a simple identity of generating functions, generalizing the argument given by Vakil-Wood \cite{vakil-wood} for computing the density of $\Conf^n X$ in $\Sym^n X$. In particular, this provides a shorter\footnote{Of course, our technique does not say anything about the Leray spectral sequence analyzed in loc. cit., and thus cannot establish any of the purely topological density results.} proof of \cite[Theorem 1.9-2]{farb-wolfson-wood}. 

\begin{remark} Ho \cite{ho:densities} has reinterpreted and extended the results of \cite{farb-wolfson-wood} using factorization cohomology. In particular, he constructs a natural rational homotopy type (a commutative dga computing the cohomology) attached to the density, then in \cite[Proposition 7.7.7]{ho:densities} obtains a simple explicit description from which one can deduce the connection with zeta values after taking the trace of Frobenius. Instead taking the characteristic power series of Frobenius, we recover the Hadamard function appearing above; thus this rational homotopy type has a meromorphic zeta function.  It would be interesting to understand this phenomenon more generally!  
\end{remark}

\subsubsection{A non-example of Hadamard convergence}\label{subsub:non-example}
We also study the density of the $k$-colored configuration spaces $C^{\vec{d}}X$ in $\Sym^{\vec{d}}X$ where $\vec{d}\in \bbZ_{\geq 0}^k$. This density converges as $\vec{d} \rightarrow \infty$ in the weight and point-counting topologies, and in Theorem~\ref{theorem:finite-label-set} we show it converges in the Hadamard topology if $k < q^{\dim X}$. Some condition of this form appears to be necessary: for $k=2$, $q=2$, and $X=\bbA^1$, we have computed the limiting formal divisor to high precision, and the result strongly suggests that the limit is not a Hadamard function --- cf. Remark~\ref{remark:hadamard-non-convergence}. 

\subsubsection{Labeled configuration spaces}
We also show Hadamard stabilization for labeled configuration spaces over unordered configuration spaces as studied in \cite{howe:mrv1} in the motivic setting and \cite{chen:0-cycles} in the point-counting setting. This does not fit into the framework of allowable labels described above, but instead admits a natural interpretation as computing the moments of a motivic random variable over unordered configuration space. Concretely, we show:
\begin{maintheorem}\label{theorem:labelled-configuration-spaces}
Let $\lambda$ be a partition and $X/\bbF_q$ a quasi-projective variety. Then, in the Hadamard topology on $\calH_1$,
\[ \lim_{d \rightarrow \infty} Z_{\Conf^{\lambda \cdot \star^d} (X)}(t) /_W Z_{\Conf^{|\lambda| + d} (X)}(t) = Z_{\Conf^\lambda_X \left( \frac{1}{1 + \bbL^{\dim X}} \right) }(t). \]
Here the right-hand-side is the zeta function of a very general notion of labeled configuration space, where the ``space" of labels at each point is the class $\frac{1}{1+\bbL^{\dim X}}$.
\end{maintheorem}

\begin{remark}\label{remark:farb-wolfson}
 	The explicit computation of the limit in Theorem~\ref{theorem:labelled-configuration-spaces} is new (though closely related to \cite[Corollary B]{howe:mrv1}), and makes precise the statement that, after passing to the zeta measure, the universal family over $C^d X$ is asymptotically a motivic binomial random variable with parameters $N= X$ and $p=\frac{1}{1+\bbL^{\dim X}}$. Convergence in the Hadamard topology, without the explicit computation of the limit, can also be deduced (under a lifting hypothesis) from the \'{e}tale homological stability results of Farb-Wolfson \cite{farb-wolfson:etale-stability} (cf. also \S\ref{subsub:betti-hadamard}). 
\end{remark}

\begin{remark}
 Following the strategy used in \cite{howe:mrv1} to relate motivic stablization of labeled configuration spaces and representation stability, one obtains the following consequence of Theorem~\ref{theorem:labelled-configuration-spaces}: given a Young diagram $\lambda$, the theory of representation stability attaches a natural sequence of locally constant $\ell$-adic sheaves $\mathcal{V}_{\lambda,d}$ on $C^d X$ for $d$ sufficiently large. Writing $L_{\lambda,d}(t)$ for the $L$-function of $\mathcal{V}_{\lambda,d}$, we find that the sequence $L_{\lambda,d}(t q^{-d \dim X})$ converges in the Hadamard topology. \end{remark}

\subsection{Batyrev-Manin over function fields}

Let $K$ be a field and $X$ a split toric variety over $K$, which is assumed smooth and projective. Let $U$ be its open orbit. For every integer $d\geq 0$, we denote by $[U_{0,d}]$ the quasi-projective variety parameterizing $K$-morphisms $\bbP^1_K \to X$ with image intersecting $U$, and of anticanonical degree $d$. Let $\rho$ be the rank of the Picard group of $X$. We are interested in the \textit{motivic height zeta function}
$$Z(T) = \sum_{d\geq 0}  [U_{0,d}] T^d.$$
In the finite field case, the specialization via point counting of $Z(T)$ has been extensively studied by Bourqui \cite{bourqui03,bourquiAMS}, in a much more general setting (for morphisms from a curve of arbitrary genus to not necessarily split toric varieties). In \cite{bourqui}, Bourqui also adressed the motivic problem over an arbitrary $K$. Combining his method therein with our results, we show:

\begin{maintheorem} \label{maintheorem:bourqui} \hfill
\begin{enumerate} \item There exists an integer $a\geq 1$ and a real number $\delta >0$ such that the series
\begin{equation}\label{eqn:motivicheightzetafunction}(1-(\LL T)^a)^\rho \left(\sum_{d\geq 0} [U_{0,d}] T^d\right)\end{equation}
converges for $|T| < \LL^{-1+\delta}$ in the dimensional topology. Its value at $\LL^{-1}$ is non-zero and can be described explicitly by the special value of a motivic Euler product. 
\item 
Assume now $K = \bbF_q$ finite. Then the specialization of (\ref{eqn:motivicheightzetafunction}) via the zeta measure converges in the point counting topology. If $q$ is larger than some explicit bound, it converges in the Hadamard topology.
\end{enumerate}
\end{maintheorem}
We refer to \S\ref{section:bourqui} and in particular to Theorem~\ref{theorem:precisebourqui} for a more precise version with explicit bounds and values. The result in the dimensional topology is obtained simply by substituting the more versatile notion of motivic Euler product from \cite{bilu:motiviceulerproducts} for the one used by Bourqui in \cite{bourqui}. The point counting convergence was already known in greater generality (for curves of any genus) by \cite{bourqui03}. The Hadamard convergence is an application of the results of the section on zero-cycles. 

This problem is an instance of the function field Batyrev-Manin conjecture (classically, the Batyrev-Manin conjecture deals with counting points of bounded height on algebraic varieties defined over number fields). As far as the authors are aware, Theorem~\ref{maintheorem:bourqui} is the first instance in the literature giving a unified treatment of a case of the function-field Batyrev-Manin problem in the point counting and motivic setting outside of situations where the motivic height zeta function is rational.

We record a consequence of Theorem~\ref{maintheorem:bourqui} on the asymptotics of the spaces $U_{0,d}$: 

\begin{corollary}\label{corollary:bourqui} For every $p\in \{0,\ldots, a-1\}$ one of the following cases occur when $d$ goes to infinity in the congruence class of $p$ modulo $a$:
\begin{itemize} \item Either $\limsup \frac{\dim U_{0,d}}{d} < 1$,
\item or $\dim U_{0,d} - d$ has  a finite limit. In this case, fixing a separable closure $K^s$ of the base field $K$ and denoting by $\kappa(U_{0,d})$ the number of irreducible components of maximal dimension of $U_{0,d}$ over $K^s$, we have that
$$\frac{\log \kappa(U_{0,d})}{\log(d)}$$
converges to an element of the set $\{0,\ldots,\rho - 1\}$. 
\end{itemize} 
Moreover, the second case happens for at least one value of $p\in \{0,\ldots,a-1\}$. 
\end{corollary}

\begin{remark} A congruence condition is in practice unavoidable: for example, in the case of projective space $\bbP^n$, the anticanonical bundle is $\mathcal{O}(n+1)$, so $U_{0,d}$ is empty when $d$ is not a multiple of $n+1$. 
\end{remark}

\subsection{Obstacles and strategies}
Our results for zero-cycles are all, in the end, obtained by explicit computations and estimates with generating functions. By contrast, in the Bertini setting which first motivated this work, similar manipulations with generating functions do not appear useful --- instead, to prove point-counting and weight stabilization results, one uses inclusion-exclusion to compare values at a finite step to truncated Euler products.  

The versions of inclusion-exclusion that come into play are quite different in the motivic and arithmetic settings, and, in particular, there does not seem to be an obvious way to merge the point-counting argument with the motivic argument in order to control the error term in the Hadamard topology. It would, however, be quite interesting if such an argument could be made! 

\subsubsection{Betti bounds and Hadamard convergence}\label{subsub:betti-hadamard}

Another angle of attack for Conjecture~\ref{conj.zeta-convergence-general} is by proving \'{e}tale cohomological stability and sub-exponential growth for the cohomology of $U_{d}$, as in the alternative proof of Hadamard convergence for labeled configuration spaces mentioned in Remark~\ref{remark:farb-wolfson}.  This approach is particularly appealing in the specific case of Conjecture~\ref{conj.zeta-convergence-Pn}, in light of Tommasi's \cite{tommasi} results on cohomological stability in characteristic zero. 

In fact, it turns out that one does not need the full strength of cohomological stability for this kind of argument: in \S\ref{section:hadamard-convergence-cohomo-stab}, we show that weight convergence combined with suitable bounds on Betti numbers implies Hadamard convergence. This seems like a promising strategy for proving new instances of our meta-conjecture.

\subsection{Organization}
In \S\ref{section:configurations} we recall some basic notation and results on (very) generalized configuration spaces, motivic Euler products, pre-$\lambda$ rings, and power structures. In \S\ref{section:hadamard_functions}, we introduce the Witt ring, its various topologies, and the zeta measure. The heart of the paper is \S\ref{section:zero-cycles}, where we prove a general convergence result on spaces of pattern-avoiding effective zero-cycles and deduce Thoerem \ref{theorem:pattern-avoiding-0-cycles}. We also discuss the case where the vectors in the set $V$ are non-orthogonal: using a M\"obius function formalism, we show Hadamard convergence over $\bbF_q$ for $q$ larger than some explicit bound, and study some interesting boundary cases. In \S\ref{section:bourqui}, we apply our results from the previous section to prove Theorem~\ref{maintheorem:bourqui}. In \S\ref{section:configuration-random-variable}, we prove Theorem~\ref{theorem:labelled-configuration-spaces}, and in \S\ref{section:hadamard-convergence-cohomo-stab} we explain the link with cohomological stability. Finally, in Appendix~\ref{appendix:computations} we give some computations related to the boundary cases for Hadamard convergence discussed in \S\ref{section:zero-cycles}. 

\subsection{Acknowledgements} The problem of identifying a reasonable setting to unify the Bertini theorems of Poonen and Vakil-Wood was posed to the first and third authors as a project by Ravi Vakil at the 2015 Arizona Winter School.  The answer we suggest here came to us only after several detours through other projects, but we thank Ravi as well as the school organizers for providing this initial opportunity, and  Ravi for his continued support along the way! We are also very grateful to Benson Farb for helpful conversations and feedback during the preparation of this work. 

During the final stages of this project, the first author was funded by the European Union's Horizon 2020
research and innovation programme under the Marie Sk\l odowska-Curie grant agreement No 893012. The third author was supported during the preparation of this work by the National Science Foundation under Award No. DMS-1704005.

\section{Recollections}\label{section:configurations} In this section we recall some basic definitions and results on generalized configuration spaces, motivic Euler products, and power structures on pre-$\lambda$ rings. 
\subsection{Grothendieck rings} In this paper $K_0(\Var/K)$, $\calM_K$, $\widehat{\calM_K}$, etc. are all built starting with the \emph{modified} Grothendieck ring of varieties (see \ref{sss.groth-ring} above); this is equivalent to the standard definition via cut and paste relations in characteristic zero but in characteristic $p$ gives a better-behaved quotient. 

We also consider, for $X/K$ a variety, the relative Grothendieck rings $K_0(\Var/X)$, $\calM_X$, and $\widehat{\calM_X}$, defined completely analogously but starting with varieties over $X$ instead of $\Spec K$.  

We refer the reader to \cite[Section 2]{bilu-howe} for more details on these points. 

\subsection{Generalized configuration spaces and motivic Euler products}\label{subsect_gen_conf_spaces_Euler_prods} We will briefly cover the basic definitions for the reader's convenience; for further discussion and properties of generalized configuration spaces and motivic Euler products beyond what is included here, we refer the reader to \cite[Sections 3.2, 6.1 and 6.2]{bilu-howe}. 
\subsubsection{Generalized configuration spaces} \label{subsect.gen_conf_spaces}
Suppose given a \emph{label set} $S$ and a finite multiset $\lambda$ supported on $S$, i.e. an element of $\bbZ_{\geq 0}^S$ that is zero on all but finitely many $s \in S$. We write $|\lambda|=\sum_{s \in S} \lambda(s).$  

For a quasi-projective variety $X/K$ we define the $\lambda$-labeled configuration space of $X$ to be
\[ C^\lambda X := \left( \left(\prod_{s \in S} X^{\lambda(s)}\right) \backslash \Delta \right) / \prod_{s} \Sigma_{\lambda(s)}. \]
where $\Delta$ is the big diagonal and $\Sigma_k$ denotes the permutation group on $k$ elements so that the product group acts in the obvious way. The points of $C^\lambda X$ in an algebraically closed field $K$ are given by labelings of $|\lambda|$ distinct points in $X(K)$ by elements of $S$ such that the total multiset of labels is equal to $\lambda$. For example, if $\lambda = (a_1, a_2, \ldots, a_k) \in \bbZ_{\geq 0}^{k} \backslash \{\vec{0}\}$, then $C^\lambda X$ is the configuration space of $a_1 + a_2 + \ldots + a_k$ distinct points on $X$ with $a_i$ of the points labeled by $i$ for each $1 \leq i \leq k$; in other words, a colored configuration space of $X$. 

The construction generalizes to allow, for each $s \in S$, a \emph{space} of labels, here interpreted to be a variety $Y_s/X$. One obtains a variety $C^\lambda \left( (Y_s/X)_{s \in S} \right)$, given by
\[ C^\lambda\left( (Y_s/X)_{s \in S} \right):= \left( \left(\prod_{s \in S} Y_s^{\lambda(s)}\right) \backslash \Delta \right) / \prod_{s} \Sigma_{\lambda(s)}, \]
 with a natural map to $C^\lambda X$ (here $\Delta$ is the inverse image of the big diagonal in the definition of $C^\lambda X$). 

\subsubsection{Motivic Euler products}

The results of \cite{bilu:motiviceulerproducts} allow one to extend the construction of generalized configuration spaces to allow the spaces of labels $Y_s$ to be replaced with classes of labels $a_s$ in a relative Grothendieck ring $K_0(\Var/X), \calM_X$, or $\widehat{\calM_X}$. The result is a class $C_X^\lambda \left((a_s)_{s \in S}\right)$ in the corresponding relative Grothendieck ring over $C^\lambda X$ (the actual definition is explained in \ref{sss.def-config} below); when $a_s = [Y_s/X]$ (i.e. the class of $Y_s$ in a relative Grothendieck ring over $X$) we have the natural identity
\[ [ C^\lambda \left( (Y_s/X)_{s \in S} \right) / C^\lambda X] = C_X^\lambda \left( (a_s)_{s \in S} \right). \]

\begin{remark} We write $C^\lambda_X(a)$ if all $a_s$ are taken to be equal to the same class~$a$. \end{remark}

In the above, if the label set $S$ is taken to be the non-zero elements of an abelian monoid $M$, then a multiset $\lambda$ as above is called a \emph{generalized partition}. In this case, it makes sense to consider the sum of its elements $\sum \lambda \in M$, and for $m \in M$ we say $\lambda \vdash m$ or $\lambda$ \emph{partitions} $m$ if $\sum\lambda=m.$ 
This setup applies in particular when $M$ is a free abelian monoid, e.g. $M=\bbZ_{\geq 0}^k$.

This extension is carried out so as to give a reasonable notion of an ``infinite product over $X$'', or, a \emph{motivic Euler product}, satisfying the natural properties one would expect for manipulating products. Indeed, for $a_i \in K_0(\Var/X)$, one defines
\[ \prod_{x\in X} \left(1 + a_{1,x} t + a_{2,x} t^2 + \ldots\right) := 1 + \sum_{m \geq 1} \left(  \sum_{\lambda \vdash m} C_X^\lambda \left( (a_s)_{s \in \bbN}\right) \right) t^m \in 1 + t K_0(\Var/K)[[t]]\]
where the sums for each coefficient are obtained by first applying the forgetful maps 
\[ K_0(\Var/C^\lambda X) \rightarrow K_0(\Var/K).\]

One can replace $K_0(\Var/K)$ here with $\calM_K$ or $\widehat{\calM_K}$. We can also make a similar construction for an abelian monoid $M$ as above by using the ring of power series in the variables $t_m$, $m \in M$, with $t_{m_1}t_{m_2}=t_{m_1+m_2}$. In this setting, the definition of the motivic Euler product becomes
\[ \prod_{x\in X} \left( 1 + \sum_{s \in M\backslash \{0\}} a_{s,x} t_s \right) := 1 + \sum_{m \in M \backslash \{0\}} \left( \sum_{\lambda \vdash m} C_X^\lambda \left((a_s)_{s \in M \backslash \{ 0\}} \right) \right) t_m. \]

Standard power series in one variable are obtained using $M=\bbZ_{\geq 0}$ via the identification $t^m = t_1^m = t_m$. More generally, for $M=\bigoplus_{i \in I} \bbZ_{\geq 0}$, we obtain power series on the variables $t_i, i \in I$. In particular, $M=\bbZ_{\geq 0}^k$ gives power series on variables $t_1, \ldots, t_k$, and because of this for $m \in \bbZ_{\geq 0}^k$ we frequently write the product $\bt^m = \prod_{1 \leq i \leq k} t_i^{m(i)}$ in place of $t_m$ in the above. 

\subsubsection{Definition}\label{sss.def-config} We briefly recall the definition of the classes $C_X^\lambda\left((a_s)_{s \in S}\right)$ used above: The first step is to define for an element $a\in K_0(\Var/X)$ (or $\calM_X$ or $\widehat{\calM_X}$) its symmetric powers $(\Sym^{n}_X(a))_{n\geq 1}$, in such a way that 
\[ \Sym^n_X(a) \in K_0(\Var/{\Sym^n(X)}), \]
and so that for all $a,b\in K_0(\Var/X)$,
\begin{equation}\label{eqn:symmectric-power-addition} \Sym^{n}_X(a + b) = \sum_{k=0}^n \Sym^k_X(a)\boxtimes\Sym^{n-k}_X(b).\end{equation}
In other words, one lifts the relative Kapranov zeta function on $K_0(\Var/X)$ so that the coefficient of $t^i$ lives in $K_0(\Var / \Sym^i X)$ rather than $K_0(\Var/X)$.

Then, for a partition $\lambda$ and classes $a_s$, we consider
 $$\Sym^{\lambda}_X((a_s)_{s\in S}):=\prod_{s} \Sym^{\lambda(s)}_X(a_s) \in K_0\left(\Var / \prod_{s} \Sym^{\lambda(s)}X\right).$$
Pulling back via the inclusion $\Conf^{\lambda}X \to \prod_{s} \Sym^{\lambda(s)}X$ gives $C^\lambda_X \left( (a_s)_{s \in S} \right).$ We often write this restriction with the subscript ``$*$'', or even ``$*,X$'' if we want to emphasize that the diagonal was removed at the level of points of $X$. If $a_s=a$ for all $s$ we also just write $a$ instead of $(a_s)$. So, e.g.,
\begin{equation}
\label{eq.config-space-def}	 C_X^\lambda (a)= \left(\Sym_X^\lambda(a)\right)_*=\left(\prod_{s} \Sym_X^{\lambda(s)}(a)\right)_{*,X}. \end{equation}
In particular, the variety $\prod_s\Sym^{\lambda(s)}X$ will be denoted $\Sym^{\lambda}X$. 
\subsection{Pre-$\lambda$ rings and power structures} 
Recall (e.g., from \cite{howe:mrv1}) that a pre-$\lambda$ ring is a ring $R$ equipped with a group homomorphism
\begin{align*} \lambda_t : (R, +) & \rightarrow (1+ tR[[t]], \times) \\
r & \mapsto 1 + \lambda_1(r)t  + \lambda_2(r)t^2 + \ldots 
\end{align*}
such that $\lambda_1(r)=r$. We require always the further condition that $\lambda_t(1)=1+t$. 

It is equivalent, and for us usually more convenient, to give the homomorphism
\[ \sigma_t : r \mapsto \lambda_{-t}(-r) = 1 + \sigma_1(r) t + \sigma_2(r) t^2 + \ldots, \]
obtained by making the substitution $t \rightarrow -t$ in $\lambda_t(-r)$. The condition $\lambda_1(r)=r$ is equivalent to $\sigma_1(r)=r$ and the condition $\lambda_t(1)=1+t$ is equivalent to $\sigma_t(1)=\frac{1}{1-t}.$ 

The operations $\lambda_i$ and $\sigma_i$ on $R$ are conveniently packaged and extended as a pairing
\[ \Lambda \times R \rightarrow R \]
for $\Lambda$ the ring of symmetric functions -- for $e_k$ the elementary symmetric functions and $h_k$ the complete symmetric functions we have 
\[ (e_k, r) = \lambda_k(r), \; (h_k, r)=\sigma_k(r), \]
and for any fixed $r \in R$ the induced map $(\bullet, r): \Lambda \rightarrow R$ is a ring homomorphism. 

\begin{example}
If $G$ is a finite group and $R$ is the complex representation ring of $G$, then $R$ is equipped with a natural pre$-\lambda$ ring structure such that, for any representation in $V$ with corresponding class $[V] \in R$ (which is also identified with the trace of $V$, viewed as a conjugation invariant function on $G$), 
\[ \lambda_k ([V])= [\textstyle{\bigwedge\nolimits^k V}],\; \sigma_k([V])=[\Sym^k V]. \]
For any $f \in \Lambda$, $(f, [V])$, viewed as a conjugation-invariant function on $G$, is the function whose value on $g \in G$ is obtained by applying $f$ to the eigenvalues of $g$ acting on $V$. 
\end{example}

\begin{example}\label{example:Kapranov-pre-lambda}
The Kapranov zeta function gives a pre-$\lambda$ ring structure on $K_0(\Var/K)$, $\calM_K$, and $\widehat{\calM_K}$ via
\[ \sigma_t([X]) = Z_X^\Kap(t). \]	
\end{example}

\begin{remark}
In categories of a combinatorial nature such as varieties or sets, one has symmetric powers but no exterior powers. However, the original formulation of (pre-)$\lambda$-rings takes places in categories of vector bundles, where exterior powers are natural. This explains why we put the emphasis on $\sigma$-operations instead of $\lambda$-operations as in classical presentations of this topic. In the literature, there is typically no restriction on $\lambda_t(1)$ and $\sigma_t(1)$ -- without this condition, one can define a new pre-$\lambda$ ring by swapping the $\sigma$ and $\lambda$-operations, so our requirements on $\lambda_t(1)$ and $\sigma_t(1)$ serve to eliminate this confusion. Our choice is the ``right one" in the sense that the operations enforced by this convention on Grothendieck rings of combinatorial and linear categories are compatible with natural functors like passing from a group action on a set to the induced permutation representation or from a variety to its compactly supported cohomology. 
\end{remark}

\begin{remark} A pre$-\lambda$ ring $R$ is a $\lambda$-ring if the map $\lambda_t$ is a pre-$\lambda$ ring homomorphism (for the Witt ring and pre-$\lambda$ structure on $1+tR[[t]]$), i.e. if it is also multiplicative and if it identifies the pre-$\lambda$ ring structures. If $R$ is torsion free over $\bbZ$, then in terms of Adams operations (see \ref{subsub.computing-simple-powers}) this is equivalent to asking that $(p_m, \bullet)$ be a ring homomorphism and $(p_{m_1}, (p_{m_2}, \bullet))=(p_{m_1m_2}, \bullet).$ The natural pre-$\lambda$ ring structure on the Grothendieck ring of a symmetric monoidal category is in fact a $\lambda$-ring structure, but it is not known whether the pre-$\lambda$ ring structure on $K_0(\Var/K)$ and its variants is a $\lambda$-ring structure. 	
\end{remark}

\subsubsection{Power structures on pre-$\lambda$ rings} In \cite{glm:power-structure, glm:power-structure2} (see also \cite{howe:mrv1}) it is explained how a pre-$\lambda$ structure on a ring $R$ extends naturally to a \emph{power structure}, which gives a systematic way to make sense of expressions like 
\[ \left(1 + \sum a_{\vec{d}}\; \bt^{\vec{d}}\right)^{b} \]
for $a_{\vec{d}}, b \in R$; the result is a new power series with coefficients in $R$ and constant term $1$, and we have
\[ (1+t)^{r}=\lambda_t(r), \; \left(\frac{1}{1-t}\right)^{r}=\sigma_t(r). \]

 In the case of $K_0(\Var/K)$ and its variants, the power structure attached to the Kapranov zeta function (viewed as a pre-$\lambda$ structure as in Example \ref{example:Kapranov-pre-lambda}) is closely related to motivic Euler products: in fact, it exactly captures the motivic Euler products with constant coefficients. Indeed, for classes $a_{\vec{d}} \in K_0(\Var/ K)$, we have 
\[ \left( 1 + \sum_{\vec{d} \in \bbZ_{\geq 0}^k \backslash \{ \vec{0}\}} a_{\vec{d}}\; \bt^{\vec{d}} \right)^{[X]}  = \prod_{x\in X} \left( 1 + \sum_{\vec{d} \in \bbZ_{\geq 0}^k \backslash \{ \vec{0}\}} a_{\vec{d}}\; \bt^{\vec{d}} \right) \]
where in the right to interpret the motivic Euler product we pull back the elements $a_{\vec{d}}$ to $K_0(\Var/X)$ as constant classes.
In particular, we obtain 
\[ Z^\Kap_X(t) = \left(\frac{1}{1-t}\right)^{[X]} = \prod_{x \in X} \frac{1}{1-t}. \]

\subsubsection{Computing simple powers}\label{subsub.computing-simple-powers}
In a useful special case, we now explain a direct explicit formula for computing these powers (or equivalently, the motivic Euler products). This formula will be used to establish several estimates in \S\ref{section:zero-cycles}. 

To that end, we consider the power sum symmetric functions
\[ p_m = x_1^m + x_2^m + \ldots \in \Lambda \]
as well as their M\"obius-inverted counterparts
\[ p_m' = \frac{1}{m}\sum_{d|m} \mu(m/d) p_d \in \Lambda[1/m] \]
studied in \cite{howe:mrv1}. We note that for $R$ a pre-$\lambda$ ring, 
\[ (p_m, \bullet): R\rightarrow R \]
is an additive homomorphism (these are the Adams operations). Indeed, this follows from the identity 
\begin{equation}\label{equation:powersum} \sum_{i\geq 0} p_{i+1}t^i = d\log \sum_{i\geq 0}{h_i}t^i.\end{equation}
and the fact that $\sigma_t$ is a homomorphism. As a consequence, we also have that 
\[ (p_m', \bullet): R[1/m]\rightarrow R[1/m] \]
is a homomorphism of additive groups.

As shown in \cite[Lemma 2.8]{howe:mrv1}, we have
\begin{lemma}\label{lemma:pre-lambda-power-formula} Suppose $f(t) \in \bbZ[[t_1, \ldots, t_n]]$ with constant coefficient $1$ and $r \in R$ for a pre-$\lambda$ ring $R$. Then, in $\left(R \otimes_{\bbZ} \bbQ\right) [[t_1, \ldots, t_n]]$, 
\[ \log \left( f(t_1, t_2, \ldots, t_k)^r \right) = \sum_{m \geq 1} p_m'(r) \log f(t_1^m, t_2^m, \ldots, t_k^m) . \]
where the exponentiation on the left-hand side is for the power structure determined by the pre-$\lambda$ structure on $R$ and $\log(1 + \ldots)$ is evaluated via the formal series
\[ \log (1+s) = s - \frac{s^2}{2} + \frac{s^3}{3} - \ldots. \]	
\end{lemma}

\section{The ring of Hadamard functions} \label{section:hadamard_functions}
\subsection{The Witt ring structure on rational functions} \label{subsect:witt_ring_rational}
We start by explaining the ring structure on the set $\calR_1$ of complex rational functions $f$ with $f(0)=1$. One can identify 
$\calR_1$ with a Grothendieck ring: let $\mathbf{Rep}_{\bbZ}$ be the the category of pairs $(V, \rho)$ where $V$ is a finite dimensional complex vector space and $\rho$ is a representation of $\bbZ$ on $V$ (to give $\rho$ is equivalent to giving the automorphism $\rho(1)$ of $V$). The characteristic power series of a linear map induces an injective map
\begin{align*}
K_0(\mathbf{Rep}_{\bbZ}) & \rightarrow 1 + t \bbC[[t]] \\
[(V, \rho)] & \mapsto \frac{1}{\det(1-t \rho(1))}
\end{align*}
with image $\calR_1$. The induced addition on $\calR_1$ is multiplication of power series, called Witt addition; the induced multiplication is called Witt multiplication, and the set $\calR_1$ equipped with this ring structure is also known as the rational Witt ring of $\bbC$. 

We note that $K_0(\mathbf{Rep}_{\bbZ})$ is also naturally isomorphic to the group ring $\bbZ[\bbC^\times]$, where the class $[a]$ in the group ring is matched with the class of the 1-dimensional representation with $\rho(1)$ given by multiplication by $a$. The induced identification of $\calR_1$ with $\bbZ[\bbC^\times]$  sends $f \in \calR_1$ to the divisor of $\frac{1}{f(1/t)}$. 

\subsubsection{The big Witt ring}\label{sss:big-witt} These ring structures extend naturally (e.g., by continuity in the coefficients) to $1+t\bbC[[t]]$; the result is the big Witt ring $W(\bbC)$. The subring 
\[ W(\bbZ) = 1+t\bbZ[[t]] \subset W(\bbC) \]
also admits a natural interpretation as the Grothendieck ring $K_0(\textbf{AlFin $\bbZ-$set})$ of the \emph{almost finite cyclic sets}  of \cite{dress-siebeneicher:burnsidering}: Here an almost finite cyclic set is a set $S$ with an action of $\bbZ$ such that the fixed points $X^{n\bbZ}$ are finite for each $n$ and $X= \bigcup_n X^{n\bbZ}.$ If we denote by $a_n(S)$ the (finite) number of orbits of length $n$ in $S$, then the identification is induced by 
\[ [S] \mapsto Z_S(t) = \prod_{n=1}^\infty \left(\frac{1}{1-t^n}\right)^{a_n(S)}. \]

Note that there is a commutative diagram
\begin{equation}\label{eqn:witt-diagram} \xymatrix{ K_0(\textbf{Fin $\bbZ-$set})\ar[d] \ar[r] & K_0(\textbf{AlFin $\bbZ-$set})=W(\bbZ)=1+t\bbZ[[t]] \ar[d] \\ K_0(\mathbf{Rep}_\bbZ)=\calR_1 \ar[r] & W(\bbC)=1+t\bbC[[t]] }.\end{equation}
Here the left vertical arrow sends a $\bbZ$-set $S$ to the permutation representation $\bbC[S]$, and its image consists of the functions in $\calR_1$ with zero and pole sets both given by unions of Galois-orbits of roots of unity, or equivalently the functions 
\[ \prod_{n=1}^\infty \left(\frac{1}{1-t^n}\right)^{a_n} \]
with $a_n \in \bbZ$ and equal to zero for $n$ sufficiently large. In particular, this can be used to show that the two ring structures on $\calR_1 \cap W(\bbZ)$ agree.

\subsection{$\lambda$-ring structure and Adams operations}\label{subsection:lambda_and_adams}
The symmetric monoidal structure on $\mathbf{Rep}_{\bbZ}$ equips the Grothendieck ring $K_0(\mathbf{Rep}_{\bbZ})$ with the structure of a $\lambda$-ring with $\sigma$-operations induced by symmetric powers and $\lambda$-operations induced by exterior powers. Under the isomorphism with~$\bbZ[\bbC^\times]$, the $\sigma$- and $\lambda$-operations are determined by
\[ \sigma_s( [a] )= 1 + \sigma_1([a])s + \sigma_2([a])s^2 + \cdots = 1 + [a]s+ [a^2] s^2 + \cdots \]
and
\[ \lambda_s( [a] ) = 1 + \lambda_1([a])s + \lambda_2([a])s^2 + \cdots = 1 + [a]s. \]

Thus, in $K_0(\mathbf{Rep}_{\bbZ})$ with the $\lambda$-ring-structure described above, (\ref{equation:powersum}) gives
$$\sum_{i\geq 0} p_{i+1}([a])s^i = \frac{\sigma_s([a])'}{\sigma_s([a])} = \frac{[a]}{1-[a]s},$$
and in particular, the Adams operations are given by $p_{i}([a]) = (p_i, [a])= [a^i].$ 

Finally, we note that symmetric powers of sets define $\sigma$-operations for a $\lambda$-ring structure on $W(\bbZ)$, and using (\ref{eqn:witt-diagram}) we find the two $\lambda$-ring structures agree on 
\[ W(\bbZ) \cap \calR_1 \subset W(\bbC). \] 

\subsection{Topologies on rational functions}\label{subsect.topologies}
\newcommand{\calW}{\mathcal{W}}
We describe three topologies on $\calR_1$.

\subsubsection{The point counting topology} \label{subsubsect:Witt-topology} There is a natural injective map
\[ \calR_1 \hookrightarrow 1+t\bbC[[t]] \]
given by taking the power series expansion at zero. The \textit{point counting topology} on~$\calR_1$ is induced by the product topology on the coefficients of
\[ 1 + t\bbC[[t]] = \bbC^{\bbN}. \]
Note that $\calR_1$ is dense when viewed as a subset of $1+t\bbC[[t]]$ so that the completion of~$\calR_1$ for the point counting topology is identified with $1 + t\bbC[[t]].$ The addition, multiplication, and $\lambda$-ring structure are continuous for the point counting topology, so that they extend to $1+ t\bbC[[t]]$ which is thus a complete topological $\lambda$-ring; this is the big Witt ring $W(\bbC)$ discussed already in \ref{sss:big-witt} above.  

Instead of taking the power-series expansion of a rational function $f$, one could instead take the power-series expansion of $d\log f$, and we would obtain the same topology. In fact, $d\log f$ gives a bijection
\[ 1 + t\bbC[[t]] \rightarrow t \bbC[[t]] = \bbC^{\bbN} \]
that is an isomorphism of topological rings when $\bbC^{\bbN}$ on the right is equipped with the product topology and ring structure. The coefficients of $d \log f$ are also called the \textit{ghost coordinates} on the big Witt ring.

Using this observation, we can also describe the point counting topology in terms of $\bbZ[\bbC^\times]=\calR_1.$ It is induced by the family of semi-norms $|| \cdot ||_j,\, j=1,2, \ldots$
\[ \left|\left| \sum_a k_a [a] \right|\right|_j = \left| \sum_a k_a a^j \right|. \]
Indeed, this follows from the above discussion and the computation
\[ d\log \left(\prod_{a}(1-ta)^{-k_a}\right) = \sum_{j=1}^\infty \left(\sum_a k_a a^j\right) t^{j-1}. \]
We note that there is no natural description of the completion of $\calR_1=\bbZ[\bbC^\times]$ for the point counting topology in terms of divisors on $\bbC^\times$.

\subsubsection{The weight topology}
In the weight topology, a basis of open neighborhoods of $f \in \calR_1$ is given by, for each $r>0$, the set of all rational functions $g$ with the same zeroes and poles as $f$ on the ball $|t| \leq r$. In particular, a sequence converges if and only if on every bounded set the zeroes and poles eventually stabilize.

Viewed as the group ring $\bbZ[\bbC^\times]$, a basis of open neighborhoods of zero is given by the set of all finite sums $\sum_{a \in \bbC^\times} k_a [a]$ supported on the closed ball of radius $r$ around $0 \in \bbC$ (here we are using that $[a]$, as a rational function, has a pole at $a^{-1}$), and a basis of open neighborhoods at any other point is given by translation.

The completion $\widehat{\bbZ[\bbC^{\times}]}^{w}$ of $\bbZ[\bbC^\times]$ for the weight topology can be described as the set of formal sums
\[ \sum_{a \in \bbC^\times} k_a [a] \]
whose support is a discrete subset of $\bbC$ and whose set of accumulation points in $\bbC \sqcup \infty$ is contained in $\{0\}$. The addition, multiplication, and $\lambda$-ring structure are all continuous for the weight topology, and extend to these formal sums.

We note that there is no natural description of the completion of $\calR_1$ for the weight topology in terms of the power series expansion at $0$.

\subsubsection{The Hadamard topology}

The Hadamard topology is most simply described under the isomorphism $\calR_1 \rightarrow \bbZ[\bbC^\times]$, where it is the topology induced by the sub-multiplicative norm
\[ \left|\left| \sum k_a [a] \right|\right|_H= \sum |k_a||a|. \]
It is easy to see the Witt and weight topologies on $\calR_1$ are not comparable (i.e., neither is finer than the other). However:
\begin{lemma} The Hadamard topology refines both the point counting and weight topologies on $\calR_1$. 
\end{lemma}
\begin{proof}
Each of the semi-norms $|| \cdot ||_j$ defining the point counting topology is continuous for the norm $|| \cdot ||_H$, and thus the Hadamard topology refines the point counting topology.

To compare with the weight topology, it suffices to observe that if $f=\sum k_a [a]$ is supported inside the closed ball of radius $r$, then so is any $g$ with $||f - g||_H < r$.
\end{proof}

\subsection{The ring of Hadamard functions}
We define the Hadamard-Witt ring $\calW$ to be the completion of $\bbZ[\bbC^\times]$ for the norm $||\cdot||_H$.  It can be identified with the set of discretely supported divisors
\[ \sum_{a \in \bbC^\times} k_a [a] \]
such that $\sum_{a \in \bbC^\times}  |k_a| |a| < \infty$.
It is an elementary computation to check that the multiplication and $\sigma$ (or $\lambda$) operations are continuous, so that they extend to $\calW$ which is thus a complete topological $\lambda$-ring.

A \emph{Hadamard} function is a meromorphic function $f$ on $\bbC$ such that $f$ can be written as a quotient $f=\frac{g}{h}$ where $g$ and $h$ are both entire functions of genus zero. In the next lemma, we extend the identification of $\bbZ[\bbC^\times]$ with $\calR_1$ to an identification of $\calW$ with the set $\calH_1$ of Hadamard functions $f$ such that $f(0)=1$.

\begin{lemma}
If $\sum_{a \in \bbC^\times} k_a [a] \in \calW$ and
\[ \sum_{a \in \bbC^\times} k_a [a] = \sum_{a \in \bbC^\times} k_a^+ [a] + \sum_{a \in \bbC^\times} k_a^- [a]\]
is the unique decomposition with $k_a^+ \geq 0$ and $k_a^- \leq 0$, then the infinite products
\[ \prod_a \left(\frac{1}{1-ta}\right)^{k_a^-} \textrm{ and } \prod_a \left(\frac{1}{1-ta}\right)^{-k_a^+} \]
converge uniformly on compact sets to entire functions of genus zero $f^-$ and $f^+$. Furthermore, the map
\[ \sum_{a \in \bbC^\times} k_a[a] \mapsto \frac{f^-}{f^+} \]
induces a bijection $\calW \rightarrow \calH_1$ extending the bijection $\bbZ[\bbC^\times] \rightarrow \calR_1$.
\end{lemma}
\begin{proof}
This is an immediate consequence of the Hadamard factorization theorem in the case of genus zero entire functions.
\end{proof}

Because the Hadamard topology refines the point counting and weight topologies, there are natural maps between the completions, and these maps have natural function-theoretic interpretations:
\begin{enumerate}
\item The map from the Hadamard completion to the point counting completion is given by taking $f(t) \in \calH_1$ to its power series at 0.
\item The map from the Hadamard completion to the weight completion is given by taking $f(t) \in \calH_1$ to the divisor of $\frac{1}{f(1/t)}$.
\end{enumerate}
The constructions in this section are summarized by the following diagram:

$$\xymatrix @R=0.6in @C=1.3in { & \calR_1 = \bbZ[\bbC^{\times}] \ar[d]_{\substack{\text{Hadamard}\\{\text{topology}}}}\ar[rd]^{\substack{\text{weight}\\{\text{topology}}}}\ar[ld]_{\substack{\text{point}\\ \text{ counting}\\ \text{topology}}} \\
W(\bbC) &\ \calH_1 = \calW\ \ar@{_{(}->}[l]^{\text{Taylor expansion at zero}}\ar@{^{(}->}[r]_{f\mapsto -\mathrm{Div}(f(t^{-1}))}&  \widehat{\bbZ[\bbC^{\times}]}^{w}\\}
$$

\subsection{The zeta measure}
Zeta functions of varieties give an interesting source of elements of the ring $\calR_1$. In fact, we have the following:
\begin{proposition}\label{prop-zeta-measure-lambda-ring}The assignment $X\mapsto Z_X(t)$ induces  a map of pre-$\lambda$-rings
$$K_0(\Var/\bbF_q)\to \calR_1,$$
where $\calR_1$ is equipped with the Witt ring structure.
\end{proposition}
\begin{proof} 
It suffices to prove the same with $\calR_1$ replaced by $1+t\bbZ[[t]]=W(\bbZ)$, because the zeta function of any variety is contained in $\calR_1 \cap 1+t\bbZ[[t]] \subset 1+t\bbC[[t]] = W(\bbC).$  Then, we claim the zeta measure is induced by the functor 
\[ \Var/\bbF_q \rightarrow \textbf{AlFin $\bbZ$-set} \]
sending $X/\bbF_q$ to the set $X(\overline{\bbF_q})$ with the action of $\mathrm{Frob}_q^{\bbZ}$.

It is clear that this induces a map of abelian groups on $K_0$: indeed, the functor factors through the localization of $\Var/\bbF_q$ at radicial surjective maps (those maps which induce bijections on points over algebraically closed fields), and then applying the naive $K_0$ to both sides gives the desired map ---  here the naive $K_0$ is the free group on isomorphism classes modulo turning finite disjoint unions into sums, and it is shown in \cite{bilu-howe} that the localization combined with naive $K_0$ on the left recovers $K_0(\Var/\bbF_q)$. Since products and symmetric powers are preserved by the functor, we find that this is furthermore a map of pre-$\lambda$ rings. 

The orbits of length $n$ in $X(\overline{\bbF_q})$ correspond to closed points of degree $n$ in $X$, and thus by definition we have 
\[ Z_X(t)=Z_{X(\overline{\bbF_q})}(t) \]
where the right-hand side is the assignment defined in \ref{sss:big-witt} inducing
\[ K_0(\textbf{AlFin $\bbZ$-set})=1+t\bbZ[[t]],\]
and we conclude. 
\end{proof}

\begin{remark} The element $\LL\in K_0(\Var/\bbF_q)$ gets sent to $\frac{1}{1-qt}$, which thanks to the way we chose our normalizations has associated element $[q]\in \bbZ[\bbC^\times]$ and therefore Hadamard norm $q$. Since $[q]$ is invertible for Witt multiplication in $\calR_1$ with inverse $[q^{-1}]$, we also see that the zeta measure induces a ring morphism
$\calM_{\bbF_q} \to  \calR_1.$
\end{remark}

\begin{remark} Let $X/\bbF_q$ be a variety. Then the different (semi-)norms introduced in Section~\ref{subsect.topologies} are expressed in the following way when applied to $Z_X(q^{-N}t)$ for some $N\geq 1$: For every $j\geq 1$, 
$||Z_{X}(q^{-N}t)||_{j} = q^{-Nj}|X(\bbF_{q^j})|$. Moreover, by Deligne's results on weights, we have 
$$||Z_{X}(q^{-N}t)||_{\infty} = q^{\dim X-N},$$
and 
 $$||Z_{X}(q^{-N}t)||_H  \leq q^{-N}\sum_{i= 0}^{2\dim X}  \dim_{\bbQ_{\ell}}H^{i}_c(X_{\bar{\bbF}_q}, \bbQ_{\ell})|q|^{\frac{i}{2}}, \; [\gcd(\ell,q)=1]$$
 with equality if $X$ is smooth projective, in which case compactly supported $\ell$-adic cohomology can also be replaced with regular $\ell$-adic cohomology.
\end{remark}

\begin{example}
It is straightforward to cook up sequences of $a_n \in \calM_{\bbF_q}$ that seemingly violate our meta-conjecture, i.e. whose zeta measures converge in the point counting and weight topologies to the same $f \in \calH_1$ but not in the Hadamard topology.
For example, take
\[a_n = (n!) \bbL^{-2n}\prod_{i=1}^n (\bbL - q^i) \in \calM_{\bbF_q}. \]
Then, $||Z_{a_n}||_\infty \rightarrow 0$ and $||Z_{a_n}||_j \rightarrow 0$ for all $j\geq 1$. Thus, in both the weight and point-counting topologies, $Z_{a_n} \rightarrow 1$, which is certainly a Hadamard function! On the other hand, 
\[  ||Z_{a_n}||_H \geq \frac{n!}{q^{n}} \]
so the sequence does not converge in the Hadamard topology. 

However we maintain that this is not a \emph{natural} sequence to consider. We refer the reader to \cref{remark:hadamard-non-convergence} for an example of a natural sequence where a different issue occurs --- there we find a sequence that converges in both the point-counting and weight topologies, but \emph{not} to a Hadamard function, so that the two limits cannot even be compared and our meta-conjecture does not apply.  
\end{example}

\subsection{The Kapranov zeta function and its special values}
As explained in the introduction, special values of Kapranov's zeta function often appear as limits in natural motivic statistics questions.

Let $X$ be a quasi-projective variety over $\bbF_{q}$, and consider the series
\begin{equation}\label{eqn.kap-zeta-series} Z^{\mathrm{Kap}}_{X,\mathrm{zeta}}(s) = 1 + Z_X(t) s + Z_{\Sym^2 X}(t) s^2 + \cdots \in 1 + s\calR_1[[s]]. \end{equation}
obtained from the usual Kapranov zeta function by applying the zeta measure.

\begin{remark}
Recall that a rational function $f\in\calR_1$ has for every $i\geq 1$ a ghost coordinate $g_i(f)$ given by the coefficient of $t^{i-1}$ in $d\log f$. As can be verified either from the series or rational function expansion of $Z^\Kap_{X,\mathrm{zeta}}(s)$, applying the $i$th ghost coordinate map~$g_i$ to each coefficient, we obtain
\[ g_i(Z_{X,\mathrm{zeta}}^\Kap(s))=Z_{X_{\bbF_{q^k}}}(s). \]
Thus, one way to think of the Kapranov zeta function of $X/\bbF_q$ is as working simultaneously with the Hasse-Weil zeta functions of the base changes of $X$ to all finite extensions of $\bbF_q$.
\end{remark}

By Proposition~\ref{prop-zeta-measure-lambda-ring}, the zeta measure is a map of pre-$\lambda$ rings,  and therefore
\[ Z^{\mathrm{Kap}}_{X,\mathrm{zeta}}(s) = \sigma_s(Z_X(t)).\]
It follows from the formula for the $\sigma$-operations on $\calR_1$ that for \emph{any} rational function $f \in \calR_1$, if we factorize
\[ f=\frac{ \prod (1-t a_i)}{\prod(1-t b_i)} \]
then
\begin{equation}\label{eqn.sigma-rational}  \sigma_s(f)= \frac{\prod( 1- s[a_i])}{\prod (1-s[b_i])} .\end{equation}
In particular, combined with Deligne's results on weights we obtain the following proposition that will be crucial in \S\ref{section:zero-cycles}:

\begin{proposition}\label{prop:kapranov-zeta-rational}
If $X/\bbF_q$ is irreducible, then 
\[ Z^{\Kap}_{X,\mathrm{zeta}}(s)\in 1 + s\calR_1[[s]] \]
is ``a rational function with smallest pole or zero given by a simple pole at $[q^{-\dim X}]$", that is,
\[  (1-[q^{\dim X}]s)Z^{\Kap}_{X,\mathrm{zeta}}(s) = \frac{\prod( 1- s[a_i])}{\prod (1-s[b_i])} \]
where the products are finite and $|a_i|, |b_i| \leq q^{\dim X - 1/2}$.
 \end{proposition}
 
In particular, this gives a way to make sense of special value in $\calH_1$ by a simple evaluation for any values of $s$ such that the denominator is invertible in $\calH_1$. The significance of knowing the locations of the smallest pole or zero is that it allows us to control the convergence of the corresponding \emph{series} expansion, which is what will come up naturally in our applications.

\section{Hadamard stabilization for effective zero-cycles}
\label{section:zero-cycles}
In this section we investigate Hadamard convergence for sequences of motivic densities arising from prescribing a set of allowable labels for effective zero cycles. After setting up the notation and problem, in Theorem~\ref{theorem:convergence-criterion} we give a general condition that guarantees these densities exist.  As immediate corollaries we obtain stabilization for the Hodge measure over $\bbC$ (Corollary~\ref{corollary:general-pattern-hodge-weight}) and for the weight topology on zeta functions (Corollary~\ref{corollary:general-pattern-zeta-weight}).  For $X$  stably rational, these results  holds already in the dimension topology on the Grothendieck ring. The proof of Theorem~\ref{theorem:convergence-criterion} is based on the generating function argument used by Vakil-Wood \cite{vakil-wood} in their study of motivic densities of configuration spaces.

We then apply our convergence criterion to obtain Hadamard convergence in various more specific settings. In \S\ref{subsec:first-app-pattern-avoiding}-\ref{sub.pattern-avoiding-II}, we consider \emph{pattern-avoiding} zero-cycles, generalizing spaces considered by Farb-Wolfson-Wood \cite{farb-wolfson-wood} and by Bourqui \cite{bourqui}. In particular, we obtain Hadamard convergence for the spaces considered in \cite{farb-wolfson-wood} as well as simpler proofs of some results of loc. cit. (cf. \S\ref{sect:orthogonal_patterns}),  and we also clarify some results of \cite{bourqui} (cf. Remark~\ref{remark:bourqui}). 

In \S\ref{subsec:second-app-conf-finite}, we consider densities with a finite set of allowable labels. The main case of interest is the universal one, corresponding to configuration spaces with a fixed set of $k$ labels. Here the behavior is more delicate: if $q^{\dim X}>k$, then we obtain Hadamard convergence, but otherwise we obtain natural examples that converge in both the weight and point counting topology, but \emph{not} to a Hadamard function.

\subsection{Notation, examples, and a general density problem}
We consider the monoid $\bbZ_{\geq 0}^k$. For a subset $A \subset \bbZ_{\geq 0}^k\setminus\{0\}$ (our set of \textit{allowable labels}) and $\vec{d} \in \bbZ_{\geq 0}^k$,  we write $\zerocycle{d}{A}$ for the configuration space of points in $X$ with labels in $A$ summing to $\vec{d}$. In other words, points in $\zerocycle{d}{A}$ can be thought of as finite formal sums $\sum_{x\in X} \vec{a}_x x$ such that $\sum_{x} \vec{a}_x = \vec{d}$. Each point of $\zerocycle{d}{A}$ thus determines a partition of $\vec{d}$ into elements of $A$, and there is a natural decomposition of $\zerocycle{d}{A}$ as a disjoint union of configuration spaces over such partitions: 
\begin{equation}\label{equation:zerocycledecomp} \zerocycle{d}{A}= \bigsqcup_{\lambda \vdash \vec{d} } C^{\lambda} X. \end{equation}
In terms of generating functions, by the definition of motivic Euler products and the decomposition (\ref{equation:zerocycledecomp}), in $\calM_K$ we have the identity 
\[ \sum_{\vec{d} \in \bbZ_{\geq 0}^k} [\zerocycle{d}{A}] \bt^{\vec{d}} = \prod_{x \in  X} \left( 1 + \sum_{\vec{a} \in A} \bt^{\vec{a}} \right) .  \]
\begin{example}\hfill \label{example:pattern-generating-functions}
\begin{enumerate}\newcommand{\bs}{\backslash}
\item Taking $A = \bbZ^{k}_{\geq 0} \setminus \{0\}$, we obtain
\[ [\zerocycle{d}{\bbZ_{\geq0}^k \setminus \{0\}}] = [\Sym^{\vec{d}}X] = [\Sym^{d_1}X \times \Sym^{d_2} X \times \cdots \times \Sym^{d_k} X] \]
and
\begin{align*}\sum_{\vec{d} \in \bbZ_{\geq 0}^k} [\zerocycle{d}{A}] \bt^{\vec{d}} &= \prod_{x \in X} \left( 1 + \sum_{\vec{a} \in A} \bt^{\vec{a}}\right) \\& = \prod_{x \in X}  \frac{1}{(1-t_1)(1-t_2)\ldots(1-t_k)}\\& = Z_X^{\Kap}(t_1) Z_X^{\Kap}(t_2) \cdots Z_X^{\Kap}(t_k). \end{align*}
In general, we may think of $\zerocycle{d}{A}$ as a constructible subset of $\Sym^{\vec{d}} X$, i.e. as a parameter space of $k$-tuples of effective zero cycles on $X$.

\item If $A$ is the collection of standard basis vectors $\vec{e}_i, 1 \leq i \leq k$, then
\[ \zerocycle{d}{A} = \Conf^{\vec{d}}(X),\]
the colored configuration space of $\sum d_i$ points in $X$ with $k$ colors and $d_i$ points of color $i$. The generating function is then
\[ \sum_{\vec{d} \in \bbZ_{\geq 0}^k} [\zerocycle{d}{A}] \bt^{\vec{d}} = \prod_{x \in X} (1 + t_1 + t_2 + \cdots + t_k). \]

\item Taking $A$ to be the complement of $\bbZ_{\geq 0}^k + (m,m,\ldots,m)$ in $\bbZ_{\geq 0}^k\setminus \{0\}$, we find that $[\zerocycle{d}{A}]$ equals the class of $k$-tuples of effective zero cycles on $X$ that overlap in a zero cycle with multiplicities less than $m$. We then have
\begin{align*} \sum_{\vec{d} \in \bbZ_{\geq 0}^k} [\zerocycle{d}{A}] \bt^{\vec{d}} &= \prod_{x \in X} \left( 1 + \sum_{\vec{a} \in A} \bt^{\vec{a}}\right) \\&= \prod_{x \in X} \frac{(1- (t_1 t_2\cdots t_k)^m)}{(1-t_1)(1-t_2)\cdots(1-t_k)}. \\
&= \frac{Z_X^{\Kap}(t_1)Z_X^{\Kap}(t_2)\cdots Z_X^{\Kap}(t_k)}{Z_X^{\Kap}\left((t_1 t_2\cdots t_k)^m\right)}. \end{align*}
\end{enumerate}
\end{example}

\subsubsection{A density problem}
Recall that a  \textit{motivic measure} is a ring morphism $\phi: \calM_K\to R$ valued in some ring $R$. We will often write $a_\phi$ for $\phi(a)$.

If $A \subset B$, then $\zerocycle{d}{A} \subset \zerocycle{d}{B}$, and for suitable motivic measures $\phi$, it is natural to consider the asymptotic density
\begin{equation}\label{equation:zero-cycle-density} \lim_{d \rightarrow \infty} \frac{[\zerocycle{d}{A}]_\phi}{[\zerocycle{d}{B}]_\phi} \end{equation}
where here $\vec{d}\rightarrow \infty$ means each entry is going to $\infty$.

We restrict here to the case where $A$ contains the standard basis vectors $\vec{e}_i$, $1 \leq i \leq k$; this ensures that $\zerocycle{d}{A}$ has dimension $\dim X \cdot \sum \vec{d}$, and includes the cases of Example \ref{example:pattern-generating-functions}. It is then natural to take $B=\bbZ^k_{\geq 0}\setminus\{0\}$, so that we are studying densities in the full motivic probability space of $k$-tuples of effective zero cycles. In this setup, we give in Theorem~\ref{theorem:convergence-criterion} below a condition on $X$, $A$, and $\phi$ that guarantees the limit (\ref{equation:zero-cycle-density}) exists, and, moreover identifies the limit as a special value of a power series expressed as a motivic Euler product. Before establishing this criterion, we give some basic notation for discussing convergence of power series with coefficients in $\calM_K$. 

\subsection{Normed motivic measures and convergence of series} 

\subsubsection{Normed motivic measures} 
We call a motivic measure $\phi:\calM_K\to R$ normed if $R$ is complete for a sub-multiplicative norm $|| \cdot ||$ and $||\LL_{\phi}^{-1}||< 1$. We have in mind especially the cases:
\begin{enumerate}
\item $K=\bbF_q$ and $\phi$ the zeta measure	 to the completion of $\calR_1$ for either the weight or Hadamard topologies. 
\item $K$ arbitrary and $\phi$ the map $\calM_K \rightarrow \widehat{\calM}_K$ to the completion for the dimension topology. In this case, we fix the norm to be $|a|=2^{\dim a}$ where we define $\dim a:=\inf\{d\in \bbZ, a\in \mathrm{Fil}_d\widehat{\calM}_K\}$.
\item $K=\bbC$ and $\phi$ the Hodge measure, with a norm defined similarly using the weight filtration on $\widehat{K_0(\HS)}$. 
\end{enumerate}
Note that on $\calM_{\bbF_q}$ we can also access the point-counting topology for the zeta measure through this setup by treating each ghost coordinate individually, i.e. by considering for each $k$ the measure induced by $X \mapsto |X(\bbF_{q^k})|$ as a $\bbC$-valued measure. 

\subsubsection{Absolute convergence and radius of convergence} The following is completely elementary, but it will be useful to spell it out clearly before we start manipulating values of convergent power series for normed motivic measures. 

If $R$ is complete for a norm $||\cdot||$, then we say a series 
\[ \sum_{\vec{d} \in \bbZ^{k}_{\geq 0}}  r_{\vec{d}} \]
for $r_{\vec{d}} \in R$ \emph{converges absolutely} if the series of real numbers $\sum ||r_{\vec{d}}||$ converges, i.e. if the limit of partial sums converges. An absolutely convergent series converges and its limit is independent of reordering the terms; moreover, the Cauchy product of two absolutely convergent series is absolutely convergent and its limit is the product of the limits of the two series. 

Given a power series 
\[ f(\bt)=\sum_{\vec{d} \in \bbZ^{k}_{\geq 0}}  a_{\vec{d}}\, \bt^{\vec{d}} \in R[[t_1, \ldots, t_k]], \]
the \emph{radius of convergence} of $f(\bt)$ is 
\[ \rho(f) = \frac{1}{\limsup_{\vec{d}} ||a_{\vec{d}}||^{1/\sum\vec{d}}} \in [0,+\infty]. \]
If $s_1, \ldots, s_k \in R$ are such that $||s_i|| < \rho(f)$ for all $i$, then the series obtained by formally substituting $s_i$ for $t_i$ converges absolutely, and we write $f(s_1, \ldots, s_k)$ for its value in $R$. 

If $f(\bt)$ and $g(\bt)$ both have radius of convergence $\geq \rho_0$, then so does the formal power series $h(\bt)=f(\bt)g(\bt)$, and $h(s_1, \ldots, s_k)=f(s_1, \ldots, s_k)g(s_1, \ldots, s_k)$ for $||s_i||<\rho_0.$ In particular, we highlight the following point: if $f(\bt)$ has invertible constant coefficient then it admits a formal inverse $\frac{1}{f}(\bt) \in R[[t_1, \ldots, t_k]].$ If both $f$ and $1/f$ have radius of convergence $\geq \rho_0$ and $||s_i|| < \rho_0$, then we find 
\[ f(s_1, \ldots, s_k) \cdot \frac{1}{f}(s_1, \ldots, s_k) = 1, \textrm{ so } \frac{1}{f}(s_1, \ldots, s_k)=\frac{1}{f(s_1, \ldots, s_k)}. \] 

\subsubsection{A useful lemma}
We establish a useful convergence lemma for the dimensional topology. Before stating it, recall that a motivic Euler product is \emph{by definition} a power series with coefficients in $\calM_K$ -- that is, the product symbol is only a notationally convenient way of defining a series. In particular, when discussing convergence of a motivic Euler product, one is always discussing convergence of a series -- there is indeed no other way that the convergence can be interpreted. 

\begin{lemma}\label{lemma:euler-product-convergence} Let $1 + \sum_{i\geq 2}a_i T^i \in \bbZ[[T]]$ be a power series with no term of degree~1. Then the motivic Euler product
$$f(t)=\prod_{x\in X} \left(1 + \sum_{i\geq 2} a_i t^i\right),$$
viewed as a power series with coeficients in $\widehat{\calM}_K$, has radius of convergence ${\geq || \LL^{-\frac{\dim X}{2}} ||.} $
Moreover, for $|s| < \LL^{-\frac{\dim X}{2}}$, $f(s)$ is an invertible element of $\widehat{\calM}_K$.
\end{lemma}
\begin{proof} Because the formal inverse 
\[ \frac{1}{1 + \sum_{i\geq 2}a_i T^i} \in \bbZ[[T]] \]
also satisfies the hypotheses of theorem and 
\[ 1/f(t) = \prod_{x \in X} \frac{1}{1 + \sum_{i\geq 2}a_i T^i}, \]
the statement about invertibility of the value $f(s)$ will follow once we have established the claim about the radius of convergence. 

To compute the radius of convergence, we note that the term of degree $n$ of the expansion of this motivic Euler product is a sum over partitions of $n$. Because $a_1 = 0$, the contribution of each partition $(n_i)$ such that $\sum i n_i = n$ is bounded in dimension by
$$\sum_{i\geq 2} n_i \dim X\leq \frac12 \sum_{i\geq 1} in_i \dim X = \frac{n\dim X}{2}.$$
The result then follows immediately from the formula for the radius of convergence. 
\end{proof}

\subsection{Weak rationality and a convergence criterion}

Given a power series $f(\bt)$ with coefficients in $\calM_{K}$ and a normed motivic measure $\phi$, the $\phi$-radius of convergence of $f$ is the radius of convergence of the power series obtained by applying $\phi$ to the coefficients of $f$. 

\begin{definition}For $X$ an irreducible $K$-variety, we say $Z_X^{\Kap}(t)$ is \emph{weakly rational} for a normed motivic measure $\phi$ if the power series $(1-\bbL^{\dim X}t)Z_X^{\Kap}(t)$ and its inverse both have $\phi$-radius of convergence strictly larger than $||\bbL^{-\dim X}_{\phi} ||$. 
\end{definition}

For $K=\bbF_q$ and $\phi$ the point-counting measure to $\bbC$, weak rationality follows from Deligne's results on weights -- indeed, the smallest zero or pole of the rational function $Z_X(t)$ is a pole of multiplicity one at $q^{-\dim X}$. For our applications, what we will need is precisely the convergence of the series expressions obtained after removing this pole, and this motivates our definition of weak rationality.

We note that closely related conditions have previously been considered in the literature. In particular, weak rationality is stronger than the motivic stabilization of symmetric powers (MSSP) of \cite{vakil-wood} and the extension $\textrm{MSSP}^*$ of \cite{howe:mrv1}.  In practice, however, weak rationality holds whenever MSSP is known, in particular:
\begin{enumerate}
\item if $X$ is stably rational then $Z_X^{\Kap}(t)$ is weakly rational for the map to the completed Grothendieck ring $\widehat{\calM}_K$,
\item for any $X/\bbF_q$, $Z_X^{\Kap}(t)$ is weakly rational for the zeta measure to $\calH_1$ by Proposition~\ref{prop:kapranov-zeta-rational}, and
\item for any $X/\bbC$, $Z_X^{\Kap}(t)$ is weakly rational for the Hodge measure.
\end{enumerate}

\subsubsection{A convergence criterion}
Adapting the strategy used to study the motivic densities of configuration spaces in \cite{vakil-wood}, we find
\begin{theorem}\label{theorem:convergence-criterion}
If $Z_X^{\Kap}(t)$ is weakly rational for $\phi$ and the power series
\begin{equation}\label{equation:convergence-criterion} \frac{\sum_{\vec{d} \in \bbZ_{\geq 0}^k} [\zerocycle{d}{A}] \bt^{\vec{d}} }{Z_X^{\Kap}(t_1)Z_X^{\Kap}(t_2)\cdots Z_X^{\Kap}(t_k)} = \prod_{x \in X} (1-t_1)(1-t_2) \cdots (1-t_k) \left(1 + \sum_{\vec{a} \in A} \bt^{\vec{a}} \right) \end{equation}
converges absolutely at $t_1=t_2=\cdots=t_k=\bbL^{-\dim X}_{\phi}$ to a value $\zeta$, then
\[ \lim_{\vec{d}\rightarrow \infty} \frac{[\zerocycle{d}{A}]_\phi}{[\Sym^{\vec{d}}X]_\phi} = \zeta. \]
\end{theorem}

\begin{proof} In what follows, we write $n = \dim X$.
By weak rationality,
\[
(1-\bbL^n t_1)\cdots(1-\bbL^nt_k)Z_X^{\Kap}(t_1)\cdots Z_X^{\Kap}(t_k) \]
converges absolutely at $\bt=(\bbL_\phi^{-n},\ldots,\bbL_\phi^{-n})$ to an invertible element. In particular, the sequence of partial sums $\frac{[\Sym^{\vec{d}} X]}{\bbL^{n \sum\vec{d}}}$ converges to an invertible element as $\vec{d}\rightarrow \infty.$

If the quotient power series
\[ \frac{\prod_{x \in  X}  (1 + \sum_{\vec{a} \in A} \bt^{\vec{a}} )}{Z_X^{\Kap}(t_1)Z_X^{\Kap}(t_2)\cdots Z_X^{\Kap}(t_k)} =  \prod_{x \in X} (1-t_1) \cdots (1-t_k) \left(1 + \sum_{\vec{a} \in A} \bt^{\vec{a}} \right)  \]
also converges absolutely at $\bt=(\bbL^{-n},\ldots,\bbL^{-n})$, then multiplying we find that
\[ (1-\bbL^nt_1)\cdots(1-\bbL^nt_k) \prod_{x \in X}\left(1 + \sum_{\vec{a} \in A} \bt^{\vec{a}} \right) \]
converges absolutely at $\bt=(\bbL^{-n},\ldots,\bbL^{-n})$. In particular, the sequence of partial sums $\frac{[\zerocycle{d}{A}]}{\bbL^{n \sum{\vec{d}}}}$  converges, and the quotient $\frac{[\zerocycle{d}{A}]}{[\Sym^{\vec{d}} X]}$ converges to the value of
\[  \prod_{x \in X} (1-t_1)(1-t_2) \cdots (1-t_k) \left(1 + \sum_{\vec{a} \in A} \bt^{\vec{a}} \right) \]
at $\bt=(\bbL^{-n}, \ldots, \bbL^{-n}).$

\end{proof}

\begin{corollary}\label{corollary:general-pattern-hodge-weight} For $X/\bbC$ irreducible,
\[ \lim_{\vec{d}\rightarrow \infty} \frac{[\zerocycle{d}{A}]_\HS}{[\Sym^{\vec{d}}X]_\HS} = \left. \left(\prod_{x \in X} (1-t_1)(1-t_2) \cdots (1-t_k) \left(1 + \sum_{\vec{a} \in A} \bt^{\vec{a}}\right) \right) \right|_{(\bbL^{-\dim X}, \bbL^{-\dim X},\ldots)}. \]
If $X$ is stably rational, then this holds already in $\widehat{\calM}_{\bbC}$.
\end{corollary}
\begin{proof} The Kapranov zeta function $Z_X^{\Kap}(t)$ is always weakly rational for the Hodge measure, and is weakly rational in $\widehat{\calM_{\bbC}}$ if $X$ is stably rational. On the other hand, since there are no degree 1 terms in \[ (1-t_1)(1-t_2) \cdots (1-t_k) \left(1 + \sum_{\vec{a} \in A} \bt^{\vec{a}}\right), \]
we can use Lemma~\ref{lemma:euler-product-convergence} to check that the motivic Euler product converges absolutely in the dimension topology on $\widehat{\calM}_{\bbC}$ and thus also for the weight topology after passing to the Hodge measure.
\end{proof}

Arguing similarly, we also obtain
\begin{corollary}\label{corollary:general-pattern-zeta-weight} For $X/\bbF_q$ irreducible,
\begin{multline*} \lim_{\vec{d} \rightarrow \infty} Z_{\zerocycle{d}{A}} /_W Z_{\Sym^{\vec{d}}(X)} = \\\left.\left(\prod_{x \in X} (1-t_1)(1-t_2) \cdots (1-t_k) \left(1 + \sum_{\vec{a} \in A} \bt^{\vec{a}}\right) \right) \right|_{([q^{-\dim X}], [q^{-\dim X}],\ldots)} \end{multline*}
in the weight topology. If $X$ is stably rational, this holds already in $\widehat{\calM}_{\bbF_q}$.
\end{corollary}

\subsection{Pattern-avoiding zero-cyles I}\label{subsec:first-app-pattern-avoiding}
A natural generalization of the setup in Example~\ref{example:pattern-generating-functions}-(3) is given by the notion of pattern-avoiding zero cycles: we fix a subset $V$ of labels, and then take $A=A(V)$ to consist of every label not lying above $V$ (where a label $\vec{a}$ lies above $V$ if there is some $\vec{v} \in V$ such that $\vec{a}-\vec{v} \in \bbZ_{\geq 0}^k$). Thus, in this case 
$\zerocycle{d}{A}$ is the space $\calZ^{\vec{d}}_{V}(X)$ considered in \cref{subsub.intro-pattern-avoiding}. In order for $A$ to contain all basis vectors, we assume here that all the vectors in $V$ have norm at least 2.

\begin{example}\label{example:pattern-avoiding-examples} \hfill
\begin{enumerate}
\item $V=(m,m, \ldots, m)$ recovers Example~\ref{example:pattern-generating-functions}-(3).
\item Let $V = \{(1,2), (2,1)\}$. Then $\zerocycle{d}{A(V)}$ parameterizes pairs of effective zero-cycles $(C_1,C_2)\in \Sym^{d_1}(X)\times \Sym^{d_2}(X)$ such that we can write $C_1 = C'_1 + D$ and $C_2 = C'_2 + D$ where $C'_1$, $C'_2$ and $D$ have disjoint supports and $D$ is reduced (i.e. all points have multiplicity at most one).
\item Bourqui \cite[Section 3]{bourqui} studied spaces of ``intersection-avoiding" zero-cycles, which corresponds to requiring each vector in $V$ to have all of its coordinates equal to 0 or 1.
\end{enumerate}
\end{example}

\subsubsection{Orthogonal patterns}\label{sect:orthogonal_patterns}
Suppose we take $V$ to be a collection of orthogonal vectors of norm larger than one (generalizing Example~\ref{example:pattern-avoiding-examples}-(1)). Then, a straightforward computation shows
\[ \left(1 + \sum_{\vec{a} \in A} \bt^{\vec{a}} \right) = \left(\prod_{i=1}^k \frac{1}{1-t_i} \right) \prod_{\vec{v} \in V} (1 - \bt^{\vec{v}}). \]
So, in this case (\ref{equation:convergence-criterion}) simplifies to
\[ \prod_{\vec{v} \in V} Z^{\Kap}_X(\bt^{\vec{v}})^{-1}. \]
In any of the normed motivic measures we consider each term in this finite product converges absolutely at $\bt=(\bbL^{-\dim X}, \bbL^{- \dim X}, \ldots)$ because $||\vec{v}|| > 1$, thus the product converges absolutely. Theorem~\ref{theorem:convergence-criterion} then gives Theorem~\ref{theorem:pattern-avoiding-0-cycles}, and we also obtain convergence in the Hodge measure for $X/\bbC$, and in the dimension topology if $X$ is stably rational. In particular, if $V=\{(m, m, \ldots, m)\}$ as in Example~\ref{example:pattern-generating-functions}-(3), then we obtain a short proof of Theorem 1.9-2 of \cite{farb-wolfson-wood} and provide the motivic lift predicted there.

\begin{remark}
Note that while the generating function for any $A$ can be written as an infinite product of zeta functions, it is quite special that we obtain a finite product in this case. 
\end{remark}

\subsection{Pattern-avoiding zero-cycles II}\label{sub.pattern-avoiding-II}
We now consider the more general case where the vectors in $V$ are not necessarily orthogonal.

\subsubsection{M\"obius functions}\label{subsub:mobius-functions} It will be helpful to rewrite the power series in Theorem~\ref{theorem:convergence-criterion} using the notion of M\"{o}bius function appearing in \cite[Section~3]{bourqui}.

\begin{definition} The local M\"obius function $\mu_V: \bbZ_{\geq 0}^k\to \bbZ$ is defined recursively by the relation
$$\mathbf{1}_{A(V)\cup \{0\}}(\ul{n}) = \sum_{0\leq \ul{n}' \leq \ul{n}}\mu_V(\ul{n}'),$$
where the left-hand side is the characteristic function of the set $A(V)\cup \{0\}$.
\end{definition}

\begin{remark}\label{remark:mu_V_values} It is immediate from the definition that
\begin{itemize} \item $\mu_V(0) = 1$.
\item  for any $\ul{n}\in A(V)$, $\mu_V(\ul{n}) = 0.$
\item for any minimal $\vec{v}\in V$, $\mu_V(\vec{v}) = -1$.
\end{itemize}
\end{remark}

\begin{lemma}\label{lemma:local_moebius_generating} We have
$$(1-t_1)\cdots (1-t_k)\left(1 + \sum_{\ul{a}\in A(V)} \bt^{\ul{a}}\right)=  \sum_{\ul{n} \in \bbZ_{\geq 0}^k} \mu_V(\ul{n})\bt^{\ul{n}}.$$
\end{lemma}

\begin{proof} This follows by dividing both sides by $(1-t_1)\cdots (1-t_k)$, expanding the right-hand side and using the definition of $\mu_V$.
\end{proof}

\begin{example} In the case where $V = \{(m,\ldots,m)\}$, we get $\mu_V(m,\ldots,m) = -1$, and $\mu_V(\ul{n}) = 0$ for any other non-zero vector $\vec{n}$, so we recover Example~\ref{example:pattern-generating-functions}-(3). More generally, if the vectors in $V$ are orthogonal, one can check that we recover the expression in Section~\ref{sect:orthogonal_patterns}.
\end{example}

It is worth showing how this notion is related to the notion of M\"obius function of a poset. For every $\vec{v}\in \bbZ^k_{\geq 0}$, we denote by $\vec{v}(i)$ its $i$-th coordinate. Define $\vec{v}_{\max}$ to be the vector given by
$$\vec{v}_{\max}(i) = \max_{\vec{v}\in V} \vec{v}(i).$$

Define $P_V:= \{0\} \cup \bigcup_{\vec{v}\in V} \left\{\ul{n},\ \vec{v}\leq \ul{n}\leq \vec{v}_{\max}\right\}.$ Then we have
\begin{proposition}\hfill \begin{enumerate} \item The restriction of the function $\mu_V$ to $P_V$ is equal to the M\"obius function of the poset~$P_V$.
\item The function $\mu_V$ is zero outside of the finite set $P_V$.
\end{enumerate}
\end{proposition}
\begin{proof} \hfill \begin{enumerate} \item 
Since for any $\ul{n}\in A(V)$, $\mu_V(\ul{n}) = 0$, the two functions satisfy the same recurrence relation for all $\ul{n}\leq \vec{v}_{\max}$, so coincide on all these elements.
\item Let $\vec{m} \in \bbZ^k_{\geq 0} \backslash P_V$. If $\vec{m} \in A(V) $, then we have already observed that $\mu_V(\vec{m})=0$. Otherwise, since $\vec{m}\not\in P_V$, we have $\vec{m} \not\leq \vec{v}_{\max}$ so there exists an index $i$ such that $\vec{m}(i) > \max_{\vec{v} \in V} \vec{v}(i)$. Note that, identifying coefficients in lemma \ref{lemma:local_moebius_generating}, we get
\[ \mu_V(\vec{m}) = \sum_{\vec{m}' \textrm{s.t.} (\vec{m}-\vec{m}') \in \{0,1\}^k} (-1)^{\sum (\vec{m}-\vec{m}')}\mathbf{1}_{A(V) \cup \{\vec{0}\}} (\vec{m}') .\]
Now, the $\vec{m}'$ in the sum break up naturally into pairs with the same coordinates outside of the $i$th index, and because for all terms we have $\vec{m}'(i) \geq \vec{m}(i)-1 \geq \max_{\vec{v} \in V} v(i)$, in each pair either both vectors or neither will lie in $A(V)$. Thus, the contributions of the two vectors in each pair cancel, and we obtain zero for the sum.

\end{enumerate}

\end{proof}

\begin{remark}\label{remark:bourqui} We comment further on Example~\ref{example:pattern-avoiding-examples}-(3). In this case, Bourqui \cite[Section 3]{bourqui} writes $B_{\mathrm{min}} \subset \{0,1\}^k$ where we write $V$, and the label generating function $1 + \sum_{\vec{a} \in A(V)} \bt^{\vec{a}}$ is what Bourqui denotes as $Q_B$; the product with the polynomial $(1-t_1)\cdots(1-t_k)$ is $P_B$, and the analogue of our Lemma~\ref{lemma:local_moebius_generating} is Bourqui's Lemme 3.1. The coefficients of (\ref{equation:convergence-criterion}) are then the values of the motivic M\"obius function of \cite[Section 3.3]{bourqui}, and the formula after specializing to Chow motives in characteristic zero in \cite[Theorem 3.3]{bourqui} follows from the identification of the motivic Euler product with a pre-$\lambda$ power and \cite[Lemma 2.8]{howe:mrv1}.

Because of the definition of motivic Euler products used there,  Bourqui's results are valid after tensoring the Grothendieck ring with $\bbQ$ and specializing to Chow motives. Our setting does not require these procedures, and can be thought of as a strengthening and generalization of the results of \cite[Section 3]{bourqui}, answering in particular Bourqui's Question 3.5: it boils down to verifying the identity

$$ \frac{\sum_{\vec{d} \in \bbZ_{\geq 0}^k} [\zerocycle{d}{A(V)}] \bt^{\vec{d}} }{Z_X^{\Kap}(t_1)Z_X^{\Kap}(t_2)\cdots Z_X^{\Kap}(t_k)}= \prod_{x\in X}\left(\sum_{\ul{n}} \mu_V(\ul{n})\bt^{\ul{n}}\right),$$
which follows from Lemma~\ref{lemma:local_moebius_generating} after taking motivic Euler products. As explained by Bourqui, a positive answer to his Question 3.5 ensures that Corollary 3.4 in loc. cit. is valid at the level of the Grothendieck ring of varieties, which in turn gives a of lift his main theorem to the Grothendieck ring of varieties. We give more details about this in Section~\ref{section:bourqui}, where we also address Hadamard convergence. 


\end{remark}

In the remainder of this subsection we prove the following convergence theorem for spaces of pattern-avoiding zero-cycles:

\begin{theorem} \label{theorem:convergence-pattern-avoiding}Suppose $X/\bbF_q$ is irreducible. Let $A = A(V)$ for some finite set of vectors $V\subset \bbZ^k_{\geq 0}$ of norms at least 2, and denote by $e$ the minimum of the sums of the coordinates of the vectors in $V$. Then,
$$Z_{\zerocycle{d}{A}} /_W Z_{\Sym^{\vec{d}}(X)}$$
converges as $\vec{d}\to \infty$ in the weight and point-counting topologies on $\bbZ[\bbC^{\times}]$. If $ \sum_{\ul{n} \in \bbZ^{k}_{\geq 0}-\{0\}}|\mu_V(\ul{n})| < q^{e\dim X}$, then it converges in the Hadamard topology.
\end{theorem}

\begin{example} Let $V = \{(2,1), (1,2)\}$ as in Example~\ref{example:pattern-avoiding-examples}-(2). Then Theorem~\ref{theorem:convergence-pattern-avoiding} tells us that $$Z_{\zerocycle{d}{A(V)}} /_W Z_{\Sym^{\vec{d}}(X)}$$ converges in the Hadamard topology for any value of $q$. Indeed, in this case $e=3$ and the only non-zero values of $\mu_V$ are 
\[ \mu_V(2,1)=\mu_V(1,2)=-1 \textrm{ and } \mu_V(2,2)=1, \]
	so the inequality becomes $3 < q^3$, which is satisfied for any prime power $q$. 
\end{example}

\begin{remark} In general, we do not necessarily expect Hadamard convergence to hold for all values of $q$. See Remark~\ref{remark:hadamard-non-convergence} for a discussion of this phenomenon, and an example of Hadamard non-convergence.
\end{remark}

To prove Theorem~\ref{theorem:convergence-pattern-avoiding}, we will apply Theorem~\ref{theorem:convergence-criterion}, and we start by establishing some bounds which will be needed to check convergence of the motivic Euler product appearing in its hypotheses. In the following we use the notation of \ref{subsub.computing-simple-powers}. 

\begin{lemma}\label{lemma:power-sum-estimates}
For $R$ a pre-$\lambda$ ring, $r \in R$ and $a_1,\ldots,a_k\in \bbZ$, the coefficient of $\mathbf{u}^{\vec{d}}$ in
\[ \log \left( (1+a_1u_1 + \cdots + a_ku_k)^r\right),  \]
an element in $R\otimes_{\bbZ} \bbQ$, is
\begin{equation}\label{eqn:coeff-exact-poly}	-\sum_{m \vec{d'}=\vec{d}} \binom{\sum \vec{d'}}{\vec{d'}}\frac{(-1)^{\sum \vec{d'}}}{\sum \vec{d'}} a_1^{d'_1}\cdots a_k^{d'_k}p_m'(r).\end{equation}
In particular, if $R \otimes_{\bbZ} \bbQ$ is normed, we find that the sum of the norms of the coefficients of $\mathbf{u}^{\vec{d}}$ for a fixed total degree $d = \sum{\vec{d}}$ is bounded by
\begin{equation}\label{eqn:coeff-estimates-poly} \sum_{m|d} \left(\sum_{i}|a_i|\right)^{d/m} ||p'_m(r)||. \end{equation}
\end{lemma}

\begin{proof}
The formula (\ref{eqn:coeff-exact-poly}) is obtained by expanding the formula given in Lemma \ref{lemma:pre-lambda-power-formula} in this case. We then obtain the estimate (\ref{eqn:coeff-estimates-poly}) by summing norms for fixed $m$ and all $\vec{d}$ in (\ref{eqn:coeff-exact-poly}).
\end{proof}

\begin{lemma}\label{lemma:convergence_poly_nolinearterms} Let $X$ be an irreducible variety over $\bbF_q$. Let $$P(t_1,\ldots,t_k) = 1 + \sum_{\vec{\imath}} b_{\vec{\imath}}\, \bt^{\vec{\imath}}\in \bbZ[t_1,\ldots,t_k]$$ be a polynomial such that $P(t_1,\ldots,t_k) - 1$ has only terms of degree at least $e\geq 2$. The motivic Euler product
$$\prod_{x\in X}P(t_1,\ldots,t_k)$$
converges absolutely at $t_1,\ldots,t_k = [q^{-\dim X}]$  in the point counting topology.  If $\sum_{\vec{\imath}}|b_{\vec{\imath}}| < q^{e\dim X}$, then it also converges absolutely in the Hadamard topology.
\end{lemma}

\begin{proof}
Since motivic Euler products commute with monomial substitutions (see \cite[Section 6.5]{bilu-howe}), and since each non-constant monomial is of degree at least $e$, we can reduce to verifying the convergence of
$$\prod_{x\in X} (1 + a_1 u_1+ \cdots + a_nu_n)$$
for $|u_i| \leq q^{-e\dim X}$, where $n$ is the number of non-constant monomials appearing in $P$, and $a_1,\ldots,a_n$ are the coefficients $b_{\ul{i}}$, arbitrarily relabeled. Because the coefficients are constant, the motivic Euler product is equivalent to the power
\[ (1 + a_1 u_1+ \cdots + a_nu_n)^{[X]} \]
Moreover, by Proposition~\ref{prop-zeta-measure-lambda-ring} the zeta measure is a map of pre-$\lambda$ rings, so we can compute the image as
\[  (1 + a_1 u_1+ \cdots + a_nu_n)^{Z_X(s)} \]
and apply Lemma \ref{lemma:power-sum-estimates}.

We first treat the Hadamard case. We extend the Hadamard norm on $\bbZ[\bbC^\times]$ to $\bbQ[\bbC^\times]$ in the obvious way; it then suffices to show that $\log$ converges absolutely. Now, we consider the estimates (\ref{eqn:coeff-estimates-poly}) for
\[ r=[Z_X(s)]=[q^{\dim X}] \pm [z_1] \pm \cdots \pm [z_N], \] where $q^{\dim X-1/2} \geq |z_i| \geq 1$. Using that $p_m'$ acts additively and that
\[ p_m'([z])=\frac{1}{m}\sum_{d|m}\mu(m/d)p_d([z])=\frac{1}{m}\sum_{d|m}\mu(m/d)[z^d], \]
(see Section~\ref{subsection:lambda_and_adams}) we obtain the estimate
\begin{equation}\label{eqn:zeta_powersum_estimate} || p_{m}'(r) ||_H \leq q^{m \dim X} (1 + N q^{-m/2}) \leq C q^{m \dim X}. \end{equation}
Suppose $\Sigma := \sum_{i}|a_i| \leq q^{e\dim X}$. Then, we can bound (\ref{eqn:coeff-estimates-poly}) by
\[ C \cdot ( \Sigma^d q^{\dim X} + \sum_{m|d, m\neq 1} \Sigma^{d/m}q^{m\dim X})\leq  C \cdot ( \Sigma^d q^{\dim X} + \sum_{m|d, m\neq 1} q^{(\frac{de}{m} + m) \dim X}).\]
$$\leq C\cdot \left( \Sigma^d q^{\dim X} + dq^{\left(\frac{de}{2} + 2\right)\dim X}\right).$$
In particular, if $\Sigma < q^{e\dim X}$, we conclude the series converges absolutely for
\[ ||u_1||_H,\ldots, ||u_k||_H \leq q^{-e\dim X}.\]

For the point counting case, it suffices to show convergence for $\bbF_q$-points. Then, $p_m'([X])$ is just an integer, the number of closed points of degree $m$ on $X/\bbF_q$. Thus for any $M$ we can factor out the polynomial
\[ \prod_{m=1}^M P(t_1^{m},\ldots,t_k^{m})^{p_m'([X])}, \]
and then taking $\log$ of what remains gives a series with coefficients bounded by $(\ref{eqn:coeff-estimates-poly})$ but where the sums are over $m > M$. We then obtain absolute convergence by taking $M$ large enough that $\Sigma^{1/M} \leq q^{e\dim X}$ and estimating as above.
\end{proof}

\begin{remark}\label{remark:hadamard-convergence-wiggleroom} In fact, as can be seen from the proof of Lemma~\ref{lemma:convergence_poly_nolinearterms}, if $\epsilon$ is such that $\Sigma < q^{e\dim X-\epsilon}$, then Hadamard convergence holds for $||u_i||_H \leq q^{-e\dim X+ \eta}$ for $\eta < \epsilon$, and thus for $||t_i||_H \leq q^{-\dim X + \eta/e}.$ In the same manner, we can get convergence of point counts for $|t_i| < q^{-\dim X + \delta}$ for some $\delta > 0$.  
\end{remark}

\begin{proof}[Proof of Theorem~\ref{theorem:convergence-pattern-avoiding}] The convergence in the weight topology follows from Corollary~\ref{corollary:general-pattern-zeta-weight}. For Hadamard and point counting convergence, we note that by the properties of the M\"obius function, the power series
$$(1-t_1)\cdots (1-t_k)\left(1 + \sum_{a\in A} \bt^{\vec{a}}\right) = \sum_{\ul{n}}\mu_V(\ul{n})\bt^{\ul{n}}$$
is a polynomial satisfying the assumptions of Lemma~\ref{lemma:convergence_poly_nolinearterms}, and the latter combined with Theorem~\ref{theorem:convergence-criterion} allows us to conclude.
\end{proof}

\subsection{Finite sets of allowable labels}\label{subsec:second-app-conf-finite}

In the previous section, we showed that in the case $A = A(V)$, we can prove Hadamard convergence of $$Z_{\zerocycle{d}{A}} /_W Z_{\Sym^{\vec{d}}(X)}$$
for $q$ sufficiently large, with the bound in $q$ depending only on the sum of the absolute values of the values of the M\"obius function of $V$.  In some special cases, it is actually possible to improve this bound. The aim of this section is to give a sharp lower bound on $q$ in the setting of Example~\ref{example:pattern-generating-functions}-(2). Thus, for a fixed $k$ we are looking at the behavior of $\Conf^{(d_1,\ldots,d_k)}(X)$ as $d_1, \ldots, d_k \rightarrow \infty$, and the corresponding generating function is
\begin{equation}\label{equation:conf-k-generating} \prod_{x \in X} (1+t_1 + \ldots + t_k). \end{equation}
\begin{remark}
For any finite set of patterns $A$ not necessarily of the form $A(V)$, the generating function for $\zerocycle{d}{A}$ is obtained by monomial substitutions from (\ref{equation:conf-k-generating}) for $k=|A|$. Similarly to the proof of Lemma~\ref{lemma:convergence_poly_nolinearterms}, the bounds we obtain in the universal case considered here can also be used to study the case of arbitrary finite~$A$.
\end{remark}

\begin{theorem}\label{theorem:finite-label-set} Suppose $X/\bbF_q$ is irreducible. Then,
\[  Z_{C^{\vec{d}}(X)} /_W Z_{\Sym^{\vec{d}}(X)} \]
converges as $\vec{d} \in \bbZ_{\geq 0}^k$ goes to infinity in the weight and point-counting topologies on $\bbZ[\bbC^\times]$. If $k < q^{\dim X}$, it converges in the Hadamard topology.
\end{theorem}

\begin{remark} This case is also covered by Theorem \ref{theorem:convergence-pattern-avoiding} by taking $V$ to be the set of all vectors in $\bbZ^k_{\geq 0}$ with sum of coordinates equal to 2.  However, the bound obtained is worse: it is of the form $f(k) < q^{2\dim X}$ for $f(k)$ exponential in $k$. For $k = 2$, however, it gives the equivalent condition $\sqrt{5} < q^{\dim X}$.
\end{remark}

\begin{remark}\label{remark:hadamard-non-convergence}
The condition for Hadamard convergence is not just an artifact of the proof: When $X=\bbA^1/\bbF_q$ and $k=2$, we can use a computer to compute the limiting formal divisor in the weight topology to high precision. The limit is of the form $\sum_{n\geq 0} (-1)^n a_n [q^{-{n}}]$, and we have verified that $a_n \geq 2^n$ for $n \leq 250$. Moreover, the computations strongly suggest that the ratios $|a_n|/|a_{n-1}|$ are a decreasing sequence for $n \geq 2$ with $\lim_{n \rightarrow \infty} |a_n/a_{n-1}|=2$. If this holds, then for $q=2$, any sequence of functions converging in the weight topology to this formal divisor has unbounded Hadamard norms. It is probably possible to prove the estimate $a_n \geq 2^n$ by expanding more carefully using the techniques below, but we leave this to the interested reader. In Appendix~\ref{appendix:computations}, we give the first 250 terms of this formal divisor; for comparison, we also give the exact divisor of $Z_{\Conf^{40,40} \bbA^1}(t q^{-80})$ along with its Hadamard and point-counting norm for $q=2$.

Note that this does not violate our meta-conjecture: limits exists for both the weight and point-counting topologies, but the limit in the weight topology is not a Hadamard function! In particular, the limit in the weight topology cannot even be compared to the the limit in the point-counting topology, because a general formal divisor does not have a well-defined Taylor expansion.
\end{remark}

To prove Theorem~\ref{theorem:finite-label-set}, we will use the same approach as for the proof of \ref{theorem:convergence-pattern-avoiding}. We start by the following variant of Lemma~\ref{lemma:power-sum-estimates}.

\begin{lemma}\label{lemma:power-sum-estimates-finite}
If $R$ is a pre-$\lambda$ ring and $r \in R$, the coefficient of $\bt^{\vec{d}}$ in
\[ \log \left(  \left((1-t_1)(1-t_2)\cdots (1-t_k)(1+t_1 + \cdots + t_k)\right)^r \right) \in R\otimes_{\bbZ} \bbQ\]
is
\begin{equation}\label{eqn:coeff-exact} \begin{cases} -\sum_{md'=d}\frac{1+(-1)^{d'}}{d'} p_m'(r) & \textrm{ if } \bt^{\vec{d}}=t_i^d  \\
					-\sum_{m \vec{d'}=\vec{d}} \binom{\sum \vec{d'}}{\vec{d'}}\frac{(-1)^{\sum \vec{d'}}}{\sum \vec{d'}} p_m'(r) & \textrm{ otherwise.} \end{cases} \end{equation}
In particular, if $R \otimes \bbQ$ is normed, we find that the sum of the norms of the coefficients of $\bt^{\vec{d}}$ for a fixed total degree $d = \sum{\vec{d}}$ is bounded by
\begin{equation}\label{eqn:coeff-estimates} \sum_{m|d, m \neq{d}} (k)^{d/m} ||p'_m(r)||, \end{equation}
\end{lemma}

\begin{proof} The proof is the same as for Lemma~\ref{lemma:power-sum-estimates}, except that in the final summation one needs to note that $m=d$ gives zero in the first case and cannot occur in the second case.
\end{proof}

\begin{proof}[Proof of Theorem~\ref{theorem:finite-label-set}]
Convergence in the weight topology follows from Corollary~\ref{corollary:general-pattern-zeta-weight}, and convergence in the point counting topology from Lemma~\ref{lemma:convergence_poly_nolinearterms} and Theorem~\ref{theorem:convergence-criterion}. For Hadamard convergence, we also apply Theorem~\ref{theorem:convergence-criterion}, and  proceed as in the proof of Lemma~\ref{lemma:convergence_poly_nolinearterms} to prove the required absolute convergence: it suffices to study convergence of the power series
\[  \left((1-t_1)(1-t_2)\cdots (1-t_k)(1+t_1 + \cdots + t_k)\right)^{Z_X(s)} \]
using Lemma~\ref{lemma:power-sum-estimates-finite}.  The point of the latter is to exploit the factors $(1-t_1)\cdots (1-t_k)$ to cancel out the contribution from  $p'_d(r)$ in (\ref{eqn:coeff-estimates}), which otherwise would have given a term which would have obstructed convergence in the estimates below. Suppose $k \leq q^{\dim X}$. Then, using the estimate (\ref{eqn:zeta_powersum_estimate}) on the M\"obius-inverted power sums of $Z_X(s)$, and the fact that there is no $m=d$ term, we may bound
(\ref{eqn:coeff-estimates}) by
\[ C \cdot ( k^d q^{\dim X} + d q^{(2+d/2)\dim X}). \]
In particular, if $k < q^{\dim X}$, we conclude the series converges absolutely for
\[ ||t_1||_H,\ldots, ||t_k||_H \leq q^{-\dim X}. \]
This concludes the Hadamard case.
\end{proof}

\begin{remark}\label{remark:l-function-stabilization}
For $X/\bbC$ smooth, one can see that the Betti numbers
\[ \dim H^i(\Conf^{(d_1, d_2, \ldots, d_k)} X(\bbC), \bbQ) = \dim H^i( \Conf^{\overbrace{(1,1,\ldots,1)}^{\sum d_i}} X(\bbC), \bbQ)^{S_{d_1} \times S_{d_2} \times \cdots \times S_{d_k}} \]
stabilize as $(d_1, d_2, \ldots, d_k) \rightarrow \infty$ using representation stability for the cohomology of pure configuration spaces combined with the Pieri rule\footnote{We thank Nate Harman for explaining to us how the Pieri rule can be applied here.}. It would be interesting to explain the growth observed in Remark~\ref{remark:hadamard-non-convergence} from this perspective.
\end{remark}

\section{Rational curves on toric varieties} \label{section:bourqui}
In this section we apply the results of Section~\ref{section:zero-cycles} to generalize the main theorem of Bourqui's paper \cite{bourqui}, which studies moduli spaces of rational curves on split toric varieties. 

\newcommand{\bbG}{\mathbb{G}}
\subsection{Geometric setting}
We now introduce the necessary notation and give a brief overview of the geometric context of the theory of (split) toric varieties. We refer to the classical references on toric varieties (e.g. \cite{fulton}) for details.

Let $K$ be a field, $r\geq 1$ be an integer, and $U = \bbG_m^r$ a split torus of dimension~$r$ defined over~$K$. We denote by $\mathcal{X}^{\ast}(U) = \mathrm{Hom}(U, \bbG_m)$ its group of characters, and $\mathcal{X}_{\ast}(U) = \mathrm{Hom}(\bbG_m,U)$ its group of cocharacters. Both are free $\bbZ$-modules of rank~$r$, and there is a natural pairing
 $$\langle \cdot,\cdot \rangle: \mathcal{X}^{\ast}(U) \times \mathcal{X}_{\ast}(U) \to \bbZ.$$
 
 A projective and regular fan $\Sigma$ of the $\bbZ$-module $\mathcal{X}_{\ast}(U)$ defines a smooth projective split toric variety $X_{\Sigma}$ with open orbit $U$. We denote by $\Sigma(1)$ the set of the rays (that is, one-dimensional faces) of $\Sigma$. A generator $\rho_{\alpha}$ of such a ray $\alpha\in \Sigma(1)$ defines a $U$-invariant divisor $D_{\alpha}$ on $X_{\Sigma}$, and there is a short exact sequence
 $$0 \to \mathcal{X}^{\ast}(U) \to \bigoplus_{\alpha \in \Sigma(1)}\bbZ D_{\alpha} \to \mathrm{Pic}(X_{\Sigma})\to 0,$$
 where the first map is given by sending $m\in \mathcal{X}^{\ast}(U) $ to $$\sum_{\alpha} \langle m,\rho_{\alpha}\rangle D_{\alpha}.$$
 From this, we in particular get the identity
\begin{equation}\label{eqn:dimension-rank-relation} \mathrm{rk} \mathrm{Pic}(X_{\Sigma}) = |\Sigma(1)| - r.\end{equation}
An anticanonical divisor is given by $\sum_{\alpha\in \Sigma(1)}D_{\alpha}$; we denote by $\calL_{0}$ its class in the Picard group. The effective cone of $X_{\Sigma}$ is the image in $\mathrm{Pic}(X_{\Sigma})\otimes \bbR$ of the cone $\sum_{\alpha} \bbR_{\geq 0} D_{\alpha}$, so that in particular $\calL_0$ lies in the interior of the effective cone of $X_{\Sigma}$. 

Bourqui's proof introduces a regular fan $\Delta$ of the $\bbZ$-module $\mathrm{Pic}(X_{\Sigma})^{\vee}$ whose support is the dual of the effective cone of $X_{\Sigma}$. The cones of maximal dimension of $\Delta$ have dimension $\rho =\mathrm{rk} \mathrm{Pic} (X_{\Sigma})$. For every ray $i\in \Delta(1)$ we denote by $m_i$ its generator. We write
\begin{equation}\label{eqn:def_of_a} a = \mathrm{lcm}\{\langle m_i,\calL_0\rangle,\ i\in \Delta(1)\}.\end{equation}
Since $\calL_0$ is in the interior of the effective cone of $X_{\Sigma}$, this is a positive integer. In this setting, the invariant $\alpha^{*}(X_{\Sigma})$ defined in Section 4.3 of Bourqui's paper may be expressed as: 
$$\alpha^{\ast}(X_{\Sigma}) = \sum_{\substack{\delta\in \Delta\\ \dim(\delta) = \mathrm{rk} \mathrm{Pic} (X_{\Sigma})}}\prod_{i\in \delta(1)} \frac{1}{\langle m_i, \calL_0\rangle}$$
(see \cite[Remarque 5.23]{bourqui}). Note that $a^{\rho}\alpha^{*}(X_{\Sigma})$ is a positive integer. 

\subsection{M\"obius functions}\label{subsec:bourqui-mobius}
By \cite[Section 3.5]{bourqui}, to every fan $\Sigma$ there is a natural way to associate a subset $B_{\Sigma}$ of $\{0,1\}^{\Sigma(1)}$ and a M\"obius function $\mu_{B_{\Sigma}}^0: \{0,1\}^{\Sigma(1)}\to \bbZ.$ Denoting by $B_{\Sigma}^{\min}$ the minimal elements of $B_{\Sigma}$ (which by Bourqui's Lemme 3.8 contains only vectors of norm at least 2), it is straightforward from the definitions that our M\"obius function $\mu_{B_{\Sigma}^{\min}}: \bbZ_{\geq 0}^{\Sigma(1)}\to \bbZ$ from Section~\ref{subsub:mobius-functions} coincides with Bourqui's $\mu_{B_{\Sigma}}^0$ on $\{0,1\}^{\Sigma(1)}$, and is zero outside of $\{0,1\}^{\Sigma(1)}$. 


We consider the elements $\mu_{\Sigma}(\ul{e}) \in \calM_K$ such that 
\begin{equation}\label{eqn:mobius-euler-product}\prod_{x\in \bbP^1}\left(\sum_{\ul{n}}\mu^{0}_{B_{\Sigma}}(\ul{n}) \bt^{\ul{n}}\right) = \sum_{\ul{e}}\mu_{\Sigma}(\ul{e})\bt^{\ul{e}}.\end{equation}
Since the answer to Bourqui's Question 3.5 is positive (see Remark~\ref{remark:bourqui}), these are analogues in the Grothendieck ring of varieties of the elements $\mu_{\Sigma}^{\chi}(\ul{e})$ considered in Bourqui's proof. 

\begin{remark} Similarly to \cite[Proposition 1-(3)]{bourqui03} and \cite[Proposition 5.18]{bourqui}, we can show, using the universal torsor formalism, that
$$\sum_{\ul{n}}\mu^{0}_{B_{\Sigma}}(\ul{n}) \LL^{-|\ul{n}|} =(1-\LL^{-1})^{\mathrm{rk}\mathrm{Pic}(X_{\Sigma})} [X_{\Sigma}]\LL^{-\dim(X_{\Sigma})}.$$
Thus, analogously to the arithmetic case, the value at $\LL^{-1}$ of our motivic Euler product (\ref{eqn:mobius-euler-product}) may be thought of as a product of local densities with convergence factors. 

\end{remark}

The main idea of the proof of Theorem~\ref{maintheorem:bourqui} will be to reduce the convergence of the motivic height zeta function to the convergence of series of the form
\begin{equation}\label{eq:general-mobius-series}\sum_{\ul{e}\in \bbZ_{\geq 0}^{\Sigma(1)}} \mu_{\Sigma}(\ul{e}) W_{\ul{e}} T^{|\ul{e}|} \end{equation}
where the $W_{\ul{e}}$ are elements in the completed Grothendieck ring of varieties. We briefly discuss here how this convergence can be checked in the different topologies in play. 

\subsubsection{Dimensional topology}\label{subsub:dimtopdiscussion}

From the proof of Lemma~\ref{lemma:euler-product-convergence}, we see that $$\dim \mu_{\Sigma}(\ul{e}) < \frac{|\ul{e}|}{2}.$$
Thus, convergence of the series (\ref{eq:general-mobius-series})  for $|T| < \LL^{-1 + \eta}$ follows as soon as one has estimates of the form
$$\dim W_{\ul{e}} \leq \epsilon |\ul{e}|$$
 for $\epsilon$ such that $0< \epsilon < \frac12$. 
In particular, this gives us an analogue in the Grothendieck ring of varieties to Bourqui's Corollaire 3.4, which is sufficient to lift Bourqui's proof to the Grothendieck ring of varieties. 
\subsubsection{Hadamard convergence}\label{subsub:hadamtopdiscussion}
Denote $M_{\Sigma} = \sum_{\ul{n}\neq 0} |\mu_{B_{\Sigma}}^0(\ul{n})|$ and let $e_{\Sigma}$ be the minimal number of non-zero coordinates of a vector in $B_{\Sigma}$.  According to Lemma~\ref{lemma:convergence_poly_nolinearterms} and Remark~\ref{remark:hadamard-convergence-wiggleroom}, if $q> M_{\Sigma}^{1/e_{\Sigma}}$, there exists $\delta>0$ such that the series 
$$\sum_{\ul{e}}\mu_{\Sigma}(\ul{e})T^{|\ul{e}|}$$
converges absolutely for $||T||_H < q^{-1 + \delta}$. We  deduce that for $\epsilon $ such that $0< \epsilon < \delta$, if we have bounds
$$||W_{\ul{e}}||_H < q^{\epsilon |\ul{e}|},$$
then the series (\ref{eq:general-mobius-series}) converges for $||T||_{H}< q^{-1 + \delta-\epsilon}$.

\subsubsection{Point counting convergence} Point counting convergence is handled similarly to Hadamard convergence: if $\delta> 0$ is such that for every prime power $q$ the series of point counts
$$\sum_{\ul{e}}\#_{\bbF_{q}}\mu_{\Sigma}(\ul{e})T^{|\ul{e}|}$$
converges for $|T| < q^{-1 +\delta}$, then it suffices to have bounds
$$\#_{\bbF_{q}} W_{\ul{e}} < q^{\epsilon |\ul{e}|}$$
for some $\epsilon$ such that $0< \epsilon < \delta$.

\subsection{Statement}
Now that we have introduced all of the data of the problem, we can state our result more precisely. 
\begin{theorem} \label{theorem:precisebourqui} 
Let $K$ be a field and $X_{\Sigma}$ a smooth and projective split toric variety over $K$ with open orbit $U$. For every integer $d\geq 0$, we denote by $U_{0,d}$ the quasi-projective variety parameterizing $K$-morphisms $\bbP^1_K \to X_{\Sigma}$ with image intersecting $U$ and of anticanonical degree $d$. Let $\rho$ be the rank of the Picard group of $X_{\Sigma}$.
\begin{enumerate} \item 
There exists a real number $\eta >0$ such that the series
\begin{equation} \label{eqn:motheightzetaprecise} (1-(\LL T)^a)^\rho \left(\sum_{d\geq 0} [U_{0,d}] T^d\right),\end{equation}
where $a$ is the integer defined in (\ref{eqn:def_of_a}), converges for $|T| < \LL^{-1+\eta}$ in the dimensional topology. Its value at $\LL^{-1}$ is non-zero and equal to

$$ a^{\rho} \alpha^*(X_{\Sigma})\LL^{r} (1-\LL^{-1})^{-\rho} \prod_{x\in \bbP^1}\left.\left(1+ \sum_{\ul{n}}\mu_{B_{\Sigma}^{0}}(\ul{n}) T^{|\ul{n}|}\right)\right|_{T = \LL^{-1}}.  $$

\item Assume now $K = \bbF_q$ finite. Then the convergence of (\ref{eqn:motheightzetaprecise}) also holds in the point counting topology. If in the notation of \ref{subsub:hadamtopdiscussion} one has $q>M_{\Sigma}^{1/e_{\Sigma}}$, then it holds in the Hadamard topology.
\end{enumerate}
\end{theorem}

\begin{remark}\label{remark:bourqui_extra_factors} The statement in \cite{bourqui} is for the series
$$(1-\LL T)^{\rho} \sum_{d\geq 0} [U_{0,d}] T^d.$$
Indeed, as we will see below, the proof consists in writing the series $Z(T)$ as a finite sum of terms of the form $C_i(T) R_i(T)$ where $C_i$ is a rational function such that $(1-\LL T)^{\rho} C_i(T)$ has no pole at $\LL^{-1}$ and $R_i(T)$ is a series which converges for $|T| < \LL^{-1 + \eta}$. Thus, while multiplying by $(1- \LL T)^{\rho}$ is enough to be able to evaluate at $\LL^{-1}$ (and thus sufficient for Bourqui's purposes), to eliminate some potential other poles of the rational functions $C_i(t)$ and obtain convergence for all $|T|< \LL^{-1+\eta}$ one needs to multiply by some additional factors.
\end{remark}
\subsection{Proof of the theorem}
The rest of the section is devoted to a proof of Theorem~\ref{theorem:precisebourqui}. This requires a careful analysis of Bourqui's proof, checking that the convergence statements can be adapted to our more general setting. Additionally to the dimensional bounds from \cite{bourqui}, we will also need some estimates from \cite{bourqui03}.

\newcommand{\mmot}{\mathrm{mot}}
Following Bourqui, we denote
$$Z_{\bbP^1,U,h_0}^{\mmot}(T) = \sum_{d\geq 0} [U_{0,d}]T^d.$$
Essentially, the proof consists in writing $Z_{\bbP^1,U,h_0}^{\mmot}(T)$ as a finite sum of series of the form (\ref{eq:general-mobius-series}), and checking convergence for each of them following the discussion in Section~\ref{subsec:bourqui-mobius}.

The first important step in Bourqui's proof is a universal torsor argument, expressing each space $[U_{0,d}]$ in terms of certain spaces of zero-cycles on $\bbP^1$. This leads to identity (5.4) in \cite{bourqui}, after which, in the beginning of section 5.4, motivic M\"obius inversion is applied. The resulting series is decomposed into a finite number of series depending on different parameters, and the contributions of which are studied separately:

\begin{equation}\label{eq:motivic-height-zeta-decomp}Z_{\bbP^1,U,h_0}^{\mmot}(T) = (\LL-1)^{-\rho} \sum_{A\subset \Sigma(1)}(-1)^{|A|} \sum_{\delta\in \Delta} Z_{A,\delta}(T)\end{equation}
where $\Delta$ is a fan with support the dual of the effective cone of $X_{\Sigma}$.  We do not give more details here, because, as Bourqui remarks in the beginning of Section 5.4, all of his computations not involving convergence issues are valid in the Grothendieck ring of varieties $\calM_{K}$.

To study convergence, one has to distinguish between different cases. We first study the case $A = \varnothing$.

For every $\delta$, there is a further decomposition 
\begin{equation} \label{eqn:zeta-empty-decomp}Z_{\varnothing,\delta} = \sum_{J\subset\Sigma(1)}(-1)^{|J|}Z_{\emptyset,\delta,J}(T).\end{equation}
We are going to show, as in Bourqui's proof, that terms $Z_{\varnothing, \delta,J}(T)$ where $J = \varnothing$ and $\delta$ is of maximal dimension (that is, $\dim(\delta) = \rho$) give the main pole.

\begin{proposition}\label{prop:A-empty} \hfill
\begin{enumerate} \item Let $\delta\in \Delta$. There exists a real number $\eta>0$ such that the series
$$(1-(\bbL T)^a)^{\dim(\delta)}Z_{\varnothing,\delta,J}(T)$$
converges for $|T|<\LL^{-1 + \eta}$ in the dimensional topology. If $K$ is finite, its specialization via the zeta measure converges in the point counting topology.  When $q> M_{\Sigma}^{1/e_{\Sigma}}$, it converges in the Hadamard topology.

\item If $J$ is nonempty and $\delta$ is of maximal dimension, the value of the series from (1) at $\LL^{-1}$ is zero.
\item The value of the series
$$(1-(\LL T)^a)^{\rho}\sum_{\substack{\delta\in \Delta\\ \dim(\delta) = r}}Z_{\varnothing, \delta, \varnothing}(T)$$
at $T = \LL^{-1}$ is
$$a^{\rho}\alpha^*(X_{\Sigma})\LL^{|\Sigma(1)|}\sum_{\ul{e}\in \bbZ_{\geq 0}^{\Sigma(1)}} \mu_{\Sigma}(\ul{e})\LL^{-|\ul{e}|},$$
which is a non-zero element of $\widehat{\calM}_K.$
\end{enumerate}
\end{proposition}

\begin{proof}
The key step in Bourqui's argument is to decompose
$$Z_{\varnothing,\delta,J}(T) = \sum_{\ul{e}}\mu_{\Sigma}(\ul{e})Z_{\varnothing,\delta,J,\ul{e}}(T)$$
and rewrite
$Z_{\varnothing,\delta,J,\ul{e}}(T)$, as in the statement of \cite[Lemme 5.26, (ii)]{bourqui}; it is a geometric series over a truncated cone denoted by $\mathcal{C}(\delta(1))_{J,\ul{e}}$, and one separates it into a product of infinite and finite geometric series.  What is important for us here is that through this procedure, we can write 
$$Z_{\varnothing,\delta,J,\ul{e}}(T) = \LL^{|\Sigma(1)|-|\ul{e}|}C_{\delta}(T)Q_{\delta,\ul{e}}(T)$$
where $C_{\delta}(T)$ is a rational function (coming from the factors which are infinite geometric series) not depending on $\ul{e}$, but only on $J$ and on $\delta$, and $Q_{\delta,\ul{e}}(T)$ is a polynomial (coming from the finite geometric sums), of the form

$$Q_{\delta,\ul{e}}(T) = \sum_{y} (\LL T)^{\langle y, \calL_0\rangle}.$$
Here the summation is over elements $y$ of  the set denoted by $\mathcal{C}(I_{J,2})_{J, \ul{e}}$ by Bourqui, the size of which is bounded by
$|\ul{e}|^{|\Sigma(1)|}$ according to the proof of Lemme 3 in \cite{bourqui03} (see bottom of page 192). Note that Bourqui's polynomial $P_{I'}(T)$ is what we denote $\LL^{|\Sigma(1)|}\LL^{-|\ul{e}|} Q_{\delta,\ul{e}}(T)$.

We have

$$(1-(\LL T)^a)^{\dim(\delta)} Z_{\emptyset, \delta,J}(T) = (1-(\LL T)^a)^{\dim(\delta)}C_{\delta}(T)\LL^{|\Sigma(1)|}\sum_{\ul{e}}\mu_{\Sigma}(\ul{e})\LL^{-|\ul{e}|}Q_{\delta,\ul{e}}(T). $$

We see from the expression $$C_{\delta}(T) = \prod_{i \in I(\delta)}\left(\frac{1}{1-(\LL T)^{\langle m_i, \calL_0\rangle}} - 1\right)$$ where $I(\delta)$ is a subset of $\delta(1)$ (the set of rays of the cone $\delta$), and from the definition of $a$, that the rational function
$$(1-(\LL T)^a)^{\dim(\delta)}C_{\delta}(T)$$
has no poles,  and we therefore may turn to the analysis of the series

$$\sum_{\ul{e}}\mu_{\Sigma}(\ul{e})\LL^{-|\ul{e}|}Q_{\delta,\ul{e}}(T).$$

Let us first consider the dimensional topology. By the top of page 193 in the proof of Lemme 3 in \cite{bourqui03}, for $y\in \mathcal{C}(I_{J,2})_{J,\vec{e}}$, we have bounds
\begin{equation}\label{eq:pairing_bounds}0\leq \langle y, \calL_0\rangle \leq C |\ul{e}|\end{equation}
for an explicit positive constant $C$, so that for $\eta>0$ sufficiently small and $|T|< \LL^{-1+\eta}$, the polynomial $Q_{\delta,\ul{e}}(T)$ takes values with dimension bounded by $\epsilon |e|$ for some  small $\epsilon>0$, and using the discussion in \ref{subsub:dimtopdiscussion} we may conclude.

We now turn to the Hadamard and point counting topologies. Using (\ref{eq:pairing_bounds}) together with the fact that the polynomial $Q_{\delta,\ul{e}}(T)$ has polynomially many terms, for $\eta$ sufficiently small and $||T||_H<  \LL^{-1+\eta}$, the values of  $Q_{\delta,\ul{e}}(T)$ are bounded by  $q^{\epsilon |e|}$ for some small $\epsilon>0$. By the discussion in \ref{subsub:hadamtopdiscussion}, we see that our series converges for $||T||_H< \LL^{-1+\eta}$ for some $\eta>0$.  We proceed similarly in the point counting case. This proves the first statement.

We now come to the second statement. From \cite[Lemme 5.26,(ii)]{bourqui}, we see that if $J$ is nonempty and $\delta$ of maximal dimension, then $I(\delta)$ is a strict subset of $\delta(1)$, so that $C_{\delta}(T)$ comprises strictly less than $\dim(\delta)$ factors, and so 
$$(1-(\LL T)^a)^{\dim(\delta)}C_{\delta}(T)$$
has a zero at $\LL^{-1}$. This together with the convergence proved above yields the result. 

It remains to prove the last statement. Assume $J$ is empty, and let $\delta$ be of maximal dimension.  In this case, from \cite[Lemme 5.26,(ii)]{bourqui}, we see that $I(\delta) = \delta(1)$ and that our polynomial $Q_{\delta,\ul{e}}$ is in fact constant equal to 1. A quick  computation then shows that the value of 
$$(1-(\LL T)^a)^{\mathrm{rk}\mathrm{Pic}(X_{\Sigma})} \sum_{\substack{\delta\in \Delta\\ \dim(\delta) =\mathrm{rk}\mathrm{Pic}(X_{\Sigma})} }C_{\delta}(T)$$
at $T = \LL^{-1}$ equals $a^{\rho}\alpha^{*}(X_{\Sigma})$. From this we deduce the value of the limit. It follows from Lemma~\ref{lemma:euler-product-convergence} that the limit is non-zero.

\end{proof}

We now come to the case $A \neq \varnothing$. The argument is similar: there is a decomposition
\begin{equation}\label{eqn:zeta-nonempty-decomp} Z_{A,\delta}(T) = \sum_{J\subset \Sigma(1)\setminus A}(-1)^{|J|}Z_{A,\delta,J}(T),\end{equation}
where
$$Z_{A,\delta,J}(T) = \sum_{\ul{e}} \mu_{\Sigma}(\ul{e})Z_{A,\delta,J,\ul{e}}(T).$$

\begin{proposition}\label{prop:A-nonempty} Let $\delta$ be a cone of $\Delta$, $A$ a non-empty subset of $\Sigma(1)$ and $J$ a subset of $\Sigma(1)\setminus A$. There exists $\eta>0$ such that the series
$$(1-(\LL T)^a) ^{\dim(\delta)} Z_{A,\delta,J}(T)$$
converges for $|T| < \LL^{-1+\eta}$ in the dimensional topology. If $k = \bbF_{q}$ is finite, it converges in the point counting topology. If moreover $q> M_{\Sigma}^{1/e_{\Sigma}}$, then it also converges in the Hadamard topology. If $\dim(\delta) = \rho$, then the value of this series at $\LL^{-1}$ is zero. 
\end{proposition}

\begin{proof} The proof proceeds similarly to the one of Proposition~\ref{prop:A-empty}. According to \cite[Lemme 5.28 (ii)]{bourqui}, we may write

$$Z_{A,\delta,T,\ul{e}}(T) = C_{\delta}(T) R_{\delta,\ul{e}}(T),$$
where $C_{\delta}(T)$ is a rational function which does not depend on $\ul{e}$, and  $R_{\delta,\ul{e}}(T)$ (denoted by $R_{I'}(T)$ in Bourqui's paper) is a power series with coefficients in $Z[\LL]$. As in the previous proposition, we again have that the rational function
$$(1-(\LL T)^a)^{\dim(\delta)} C_{\delta}(T)$$
has no poles, and has a zero at $\LL^{-1}$ if $\delta$ is of maximal dimension. Thus, essentially the only difference with the previous case is that the polynomial factor $Q_{\delta, \ul{e}}$ has been replaced with a non-polynomial one, and we need to work a little bit more to get sufficient bounds. 

%

Writing $R_{\delta,\ul{e}}(T)$ out explicitly, we see that we are interested in the convergence properties of the series

\begin{equation}\label{eqn:mobius_series_Anonempty} \sum_{\ul{e}} \mu_{\Sigma}(\ul{e}) \LL^{-|\ul{e}|} \sum_{\substack{(h_{\alpha})_{\alpha\in A} \\ h_{\alpha} \geq e_{\alpha}}} \LL^{-\sum_{\alpha\in A}(h_{\alpha} - e_{\alpha})}\sum_{y} (\LL T ) ^{\langle y, \calL_{0}\rangle} \end{equation}
where the sum over $y$ is taken over a finite subset of the dual of the effective cone, the size of which is bounded polynomially in the $h_{\alpha}$ and $e_{\alpha}$, according to the middle of page 197 in the proof of Lemme 4 in \cite{bourqui03}. In the same reference, we also see that for $y$ in this set, there is a positive constant $C$ such that

$$0 \leq \langle y,\calL_{0}\rangle \leq C \left(\sum_{\alpha\in A}(h_{\alpha}-e_{\alpha}) + |\ul{e}|\right).$$ 

Using the latter, we see that for $|T|< \LL^{-1+\eta}$ the dimension of the term corresponding to $(h_{\alpha})$ is bounded by 
$$-(1- \epsilon)\left(\sum_{\alpha\in A}(h_{\alpha}-e_{\alpha})\right) + \epsilon|\ul{e}|$$ for some small $\epsilon$. From this we see that the $\ul{e}$-term of (\ref{eqn:mobius_series_Anonempty}) converges in the dimensional topology for $|T| < \LL^{-1 + \eta}$ and takes values with dimension bounded by $\epsilon |\ul{e}|$. By \ref{subsub:dimtopdiscussion}, we have the desired convergence of the series (\ref{eqn:mobius_series_Anonempty}) in the dimensional topology. 

In the same way, the Hadamard norm of the $\ul{e}$-term of (\ref{eqn:mobius_series_Anonempty}) is bounded by
$$\sum_{\substack{(h_{\alpha})_{\alpha\in A} \\ h_{\alpha} \geq e_{\alpha}}} q^{-(1-\epsilon)\sum_{\alpha\in A}(h_{\alpha} - e_{\alpha})}q^{\epsilon |\ul{e}|} = \left(\frac{1}{1-q^{-1+\epsilon}}\right)^{|A|}q^{\epsilon |\ul{e}|}$$
which by \ref{subsub:hadamtopdiscussion} is enough to deduce Hadamard convergence. Point counting convergence is handled in the same way. 
\end{proof}

To conclude the proof of Theorem~\ref{theorem:precisebourqui}, we combine the decompositions in (\ref{eq:motivic-height-zeta-decomp}), (\ref{eqn:zeta-empty-decomp}) and (\ref{eqn:zeta-nonempty-decomp}) with Propositions~\ref{prop:A-empty} and~\ref{prop:A-nonempty}, to show that 

$$(1-(\LL T)^a)^{\mathrm{rk}\mathrm{Pic}(X_{\Sigma})} Z_{\bbP^1, U, h_0}^{\mmot}(T)$$
is a finite sum of series that converge for $|T|< \LL^{-1+\eta}$. Moreover, the only ones that give a non-zero contribution to the value at $\LL^{-1}$ are those corresponding to $A = \varnothing$, $J = \varnothing$ and $\delta$ of maximal dimension, and their contribution is given by Proposition~\ref{prop:A-empty}-(3). Using relation (\ref{eqn:dimension-rank-relation}), we conclude that the value of our series at $T = \LL^{-1}$ is

$$a^{\rho}\alpha^*(X_{\Sigma})\LL^{r}(1-\LL^{-1})^{-\rho} \sum_{\ul{e}\in \bbZ_{\geq 0}^{\Sigma(1)}} \mu_{\Sigma}(\ul{e})\LL^{-|\ul{e}|}.$$

\begin{remark}\label{remark:limit_dominant_term}
In fact, the proof of Lemma~\ref{lemma:euler-product-convergence} allows us to deduce that this value is of the form $$a^{\rho}\alpha^{*}(X_{\Sigma})\LL^{\rho} + \text{terms of lower dimension},$$
where we recall that $a^{\rho}\alpha^{*}(X_{\Sigma})$ is a positive integer. 
\end{remark}
\begin{proof}[Proof of Corollary~\ref{corollary:bourqui}] We may apply Proposition 4.7.3.1 from \cite{bilu:motiviceulerproducts}. While it is stated over $\bbC$ and in the topology of weights used in loc.cit., its proof is valid over any field $K$ replacing weight with dimension multiplied by two, and using the Euler-Poincar\'{e} polynomial (see \cite[Chapter 2 -- (3.5.9), Proposition 3.5.10 and Corollary 3.5.12]{chambert-loir-et-al:motivic-integration}) instead of the Hodge-Deligne polynomial. The condition on the effectivity of the value at $\LL^{-1}$ can be replaced by the contents of Remark~\ref{remark:limit_dominant_term}, because it is only used in the proof to produce non-zero coefficients for the corresponding Hodge-Deligne polynomial. 
\end{proof}
\section{The configuration random variable} \label{section:configuration-random-variable}
In this section we prove the following generalization of Theorem~\ref{theorem:labelled-configuration-spaces}, which also strengthens and provides a more natural formulation of \cite[Corollary B]{howe:mrv1}. 
\begin{theorem}\label{theorem:conf-rv-body} Let $X$ be an irreducible variety over a field $K$. If 
\begin{enumerate}
\item $K$ is arbitrary, $X$ is stably rational, and $\phi$ is the measure to $\widehat{\calM_K}$ or
\item $K=\bbC$ and $\phi$ is the Hodge measure to $\widehat{K_0(\HS)}$, or
\item $K=\bbF_q$ and $\phi$ the zeta measure to $\calH_1$,
\end{enumerate}
then
\[ \lim_{d\to \infty}\frac{[C^{\lambda \cdot *^d} (X)]_{\phi}}{[C^{|\lambda|+d}(X)]_{\phi}}  = C^{\lambda}_{X} \left( \frac{1}{1 + \bbL_\phi^{\dim X}}\right). \]
\end{theorem}
Here the element $ C^{\lambda}_{X} \left( \frac{1}{1 + \bbL_\phi^{\dim X}}\right)$ will be made sense of by showing that the map 
\[ a \mapsto C^\lambda_X(a)_\phi \]
extends by continuity to the closure of $\bbZ[\bbL^{\pm1}]\cong \phi(\bbZ[\bbL^{\pm1}])$, which contains
\[ \frac{1}{1+\bbL_\phi^{\dim X}} = \bbL_\phi^{-\dim X}\frac{1}{1+\bbL_\phi^{-\dim X}}=\bbL_\phi^{-\dim X} -\bbL_\phi^{-2\dim X} +.... \]

We also note that in case (1) of the theorem, the assumption that $X$ is stably rational is only there to ensure $Z^\Kap_X(t)$ is weakly rational, and could be replaced with that condition.  

\subsection{Continuity of labeled configuration spaces} We now prove a lemma giving the continuity properties required to make reasonable sense of $C^\lambda\left(\frac{1}{1+\bbL_\phi^{\dim X}}\right)$ and similar quantities. In the dimension topology (and thus also for the Hodge measure) a very strong continuity on the entire Grothendieck ring follows immediately from the definitions. The case of the Hadamard topology is more subtle because we do not know any suitable general bounds on the Hadmard norm of a labeled configuration space. However, if we restrict to $\bbZ[\bbL^{\pm 1}]$, which is enough for our purposes here, a simple estimate will suffice. 
  
\begin{lemma}\label{lemma.continuity-configuration-spaces} Suppose $X\rightarrow S$ is a map of varieties over a field $K$. 
\begin{enumerate}
\item The map
\[ \calM_X \rightarrow \calM_{C^\lambda(X)},\; a \mapsto C^\lambda_X(a) \]
is continuous for the dimension topologies on both sides and thus induces a continuous map
\[ \widehat{\calM_X} \rightarrow \widehat{\calM_{C^\lambda(X)}},\; a \mapsto C^\lambda_X(a). \]
\item If $K=\bbC$ and $\phi$ is the Hodge measure to $\widehat{K_0(\HS)}$, then composition of the arrow from (1) with the forgetful map $\widehat{\calM_{C^\lambda(X)}} \rightarrow \widehat{\calM_\bbC}$ and $\phi$ induces a continuous map 
\[ \widehat{\calM_X} \rightarrow \widehat{K_0(\HS)},\; a \mapsto C^\lambda_X(a)_\phi. \]
\item If $K=\bbF_q$ and $\phi$ is the zeta measure, the map 
\[ \bbZ[\bbL^{\pm}]=\phi(\bbZ[\bbL^{\pm}])\rightarrow \calH_1,\; a\mapsto C^\lambda_X(a)_\phi =Z_{C^\lambda_X(a)}(t) \]
extends to a continuous map
\[ \widehat{\bbZ[\bbL^{\pm 1}]}=\overline{\phi(\bbZ[\bbL^{\pm1}])} \rightarrow \calH_1,\; a \mapsto C^\lambda_X(a)_\phi \]
where the domain is the completion of $\bbZ[\bbL^{\pm 1}]$ for the induced norm, or equivalently the closure of $\phi(\bbZ[\bbL^{\pm1}])$ in $\calH_1$. 
\end{enumerate}
\end{lemma}
\begin{proof}
The first two statements are immediate from the definitions via simple dimension estimates. We now treat the third statement; in the proof, we denote the Hadamard norm by $|| \cdot ||$.

Fix $a \in \bbZ[\bbL^{\pm 1}] $, and consider a perturbation $a + h$. 
We write $h$ in the slightly unusual expansion $h= \sum_{i=1}^N \epsilon_i \bbL^{k_i}$ 
for $\epsilon_i \in \{\pm 1\}$ and $\epsilon_i=\epsilon_j$ if $k_{i}=k_{j}$; for example, we expand 
\[ 2 \bbL^2 - 3 \bbL^{-1} = \bbL^2 + \bbL^2 - \bbL^{-1} - \bbL^{-1} - \bbL^{-1}. \]
Note that as a consequence, we may express the Hadamard norm of $h$ in the following way:
$$||h|| = \sum_{i=1}^N ||\LL||^{k_i}.$$ 
Using the definition (\ref{eq.config-space-def}) of $C^\lambda_X$, the symmetric power addition formula (\ref{eqn:symmectric-power-addition}), and the identity $\Sym_X^k(\bbL^\ell)=\bbL^{k\ell}$ valid in any relative Grothendieck ring, we find
\[ C^\lambda_X(a+h)- C^\lambda_X(a) =\sum_{\substack{\sum_{j=0}^N\lambda_j = \lambda \\ \lambda_0 \neq \lambda}} \left(\Sym_X^{\lambda_0}(a)\prod_{j=1}^N\Sym^{\lambda_j}_X (\epsilon_i)\right)_{*,X} \bbL^{\sum_{j=1}^N |\lambda_j|k_j} \]
Here the sum is over tuples of partitions $(\lambda_0,\ldots,\lambda_N)$ such that their multiplicity vectors sum up to the multiplicity vector of $\lambda$, and such that the partition $\lambda_0$ is not equal to $\lambda$. 

The key point is that, since $\lambda$ and $a$ are fixed, the terms $(\ldots)_{*,X}$ appearing vary over a finite set of classes (for all $h$). We can thus bound their Hadamard norms above by a real number $M$, so that we obtain 
\begin{align*} ||C^\lambda_X(a+h )- C^\lambda(a)|| & \leq M \sum_{\substack{\sum_{j=0}^N\lambda_j = \lambda \\ \lambda_0 \neq \lambda}} ||\bbL||^{\sum_{j=1}^N |\lambda_j|k_j} \\
& \leq M \left( \left(1+\sum_{j=1}^N ||\bbL||^{k_j}\right)^{|\lambda|} - 1 \right) \\
&= M \left( \left(1+||h||\right)^{|\lambda|} - 1 \right)
 \end{align*}
To go from the first to the second line, observe that if the multiplicities of $\lambda$ are $m_1, \ldots, m_n$, then
\[ \sum_{\substack{\sum\lambda_j = \lambda \\ \lambda_0 \neq \lambda}} ||\bbL||^{\sum_{j=1}^N |\lambda_j|k_j} = \left(\prod_{l=1}^n\left( \sum_{\sum_{j=1}^N a_j\leq m_l} ||\bbL||^{\sum a_j k_j} \right)\right) - 1 \] 
and the $l$-th term inside the product is bounded above by 
\[\left( 1+\sum_{j=1}^{N}||\bbL||^{k_j}\right)^{m_l} = \sum_{\sum_{j=1}^N a_j\leq m_l}\binom{m_l}{a_1,\, a_2, \ldots, a_N,\, m_l - \sum_{j=1}^N a_j} ||\bbL||^{\sum_{j=1}^N a_j k_j}. \]

This verifies continuity at $a$, and we conclude since $a$ was arbitrary.  
\end{proof}

 \subsection{Configuration spaces with power series labels}\label{subsect.power_series_conf_spaces}
 Let now 
 \[ f(s) = a_0 + a_1 s + \ldots \in \calM_X[[s]] \]
 Using property (\ref{eqn:symmectric-power-addition}),  for any generalized partition $\lambda = (n_i)_i$, there is a natural way of defining a power series
 $$\Conf^{\lambda}_X(f(s))\in \calM_{\Conf^{\lambda}(X)}[[s]].$$
 Explicitly, we have, denoting by $()_{*,X}$ the pullback to $\Conf^{\lambda}(X)$,

\begin{eqnarray}\nonumber \Conf^{\lambda}_X(f(s)) &=&\left( \prod_{j\geq 1} \Sym^{n_j}_X\left(\sum_{i\geq 0} a_is^i\right) \right)_{*,X} \\
\nonumber &=& \left(\prod_{j\geq 1}\left( \sum_{\substack{(n_{i,j})_{i}\\ \sum_in_{i,j} = n_j}} \left(\prod_i \Sym^{n_{i,j}}_X(a_i)\right)s^{\sum_i in_{i,j}}\right)\right)_{*,X}\\
\label{eq.config-power-series-exp} & = &\sum_{\substack{(n_{i,j})_{i,j}\\ \sum_in_{i,j} = n_j}} \left( \prod_{i,j}  \Sym^{n_{i,j}}_X(a_i)\right)_{*,X} s^{\sum_{i,j} i n_{i,j}}.
 \end{eqnarray}

Arguing similarly to our proof of Lemma \ref{lemma.continuity-configuration-spaces}, we find

\begin{lemma}\label{lemma:values_conf_series} In the settings of Lemma \ref{lemma.continuity-configuration-spaces}, if $f(s) \in \bbZ[[s]]$ and if $f$ converges absolutely at $\bbL_{\phi}^r$, then so does $C^\lambda_X(f(s))$ and 
\[ \Conf^{\lambda}_X(f(\LL_\phi^{r})) = \left.\Conf^{\lambda}_X(f(s))\right|_{s = \LL_\phi^{r}}. \]
\end{lemma}

In the dimension topology, the statement holds for $f(s) \in \calM_X[[s]],$
but we will not use this added generality here.

\begin{remark}\label{eqn:Euler_product_f} By expanding, one can see that the coefficient of $t^{\lambda}$ in the motivic Euler product
$$\prod_{x\in X} (1 + f_x(s)(t_1 + t_2 + \cdots)) = \prod_{x\in X} \left( 1+ \sum_{i,j}a_{i,x}s^it_j\right)$$
is exactly the image in $\calM_K[[s]]$ of $\Conf^{\lambda}_X(f(s))$.
\end{remark}

\subsubsection{An alternative expression for configuration spaces with power series labels}
To motivate what we want to establish in this section, let us discuss quickly the classical set-up that we are trying to imitate. When $X$ is a finite set and $(f_x(s))_{x\in X}$ is a family of formal power series indexed by $X$, the expansion of the finite product
\begin{equation}\label{eqn:finite_product_configuration}\prod_{x\in X} (1 + f_x(s)t)\end{equation}
can be written as
$$\sum_{n\geq 0} \left(\sum_{c\in C^n(X)}\prod_{x\in c} f_x(s)\right) t^n,$$
where $C^n(X)$ is the set of configurations of $n$ distinct points of $X$. In other words, the family $(f_x(s))_{x\in X}$ defines a function on $C^n(X)$ given by $c\mapsto \prod_{x\in c}f_x(s)$, and the coefficient of $t^n$ in the expansion of (\ref{eqn:finite_product_configuration}) is the summation of this function over $C^n(X)$. 

In the usual Grothendieck ring dictionary, elements of $\calM_{\Conf^{\lambda}(X)}$ can be thought of as motivic functions defined on $\Conf^{\lambda}(X)$, and taking the class of such an element in $\calM_K$ may be thought of as summation over $\Conf^{\lambda}(X)$. In view of Remark~\ref{eqn:Euler_product_f}, if one replaces finite products with motivic Euler products, one should expect $C^{\lambda}_X(f(s))$ to be equal in $\calM_{\Conf^{\lambda}(X)}$ to a motivic Euler product relatively to $\Conf^{\lambda}(X)$: to reproduce the fact that above every configuration we take the product over points of that configuration, the product will be over the universal configuration.

For $X$ a variety over $K$ and $\lambda$ a partition, let $\bc_{\lambda} / \Conf^{\lambda} X$ denote the universal configuration,
\[ \bc_{\lambda}  = \{(c, x) | x \in c\} \subset \Conf^{\lambda} X \times X\]
Denote by $j_{\lambda} : \bc_{\lambda} \rightarrow X$ the projection. Given $f(s) \in \calM_{X}[[s]]$, let $j_{\lambda}^* f$ be the corresponding series in $\calM_{\bc_{\lambda}}[[s]]$ given by pullback of coefficients along $j_{\lambda}$.

\begin{proposition} \label{prop:power-series-conf-space}We have the equality
$$C^{\lambda}_X(f(s)) = \prod_{y\in \bc_{\lambda}/C^{\lambda} X}(j_{\lambda}^*f)_y(s)$$
in $\calM_{\Conf^{\lambda}(X)}[[s]].$
\end{proposition}

\begin{proof} We start by expanding the right-hand side. For every $i$, denote $b_i = j_{\lambda}^{*}a_i$. For every $j\geq 1$, there is a projection map
$$\pi_j:\Conf^{\lambda} X = \left(\prod_{i\geq 1}\Sym^{n_i}X\right)_{*,X} \to (\Sym^{n_{j}}X)_{*,X},$$ using which we introduce
$$\bc_{\lambda}^{(j)} =\{(x,c)\in \bc_{\lambda},\ x\in \pi_j(c)\}.$$
By definition, $\bc_{\lambda}$ is the disjoint union of the $\bc_{\lambda}^{(j)}$, $j\geq 1$. We also define $b_i^{(j)}$ to be the restriction of $b_i$ to $c_i^{(j)}$, so that in $\calM_{\bc_{\lambda}}$, we have
$$b_i = \sum_{j} b_{i}^{(j)}.$$
In other words, $b_i^{(j)}$ is the pullback of $a_i$ to $\bc_{\lambda}^{(j)}$. We now expand
\begin{eqnarray*} \prod_{y\in \bc_{\lambda}/\Conf^{\lambda} X}(j^{*}_{\lambda}f)_x (s)& = & \prod_{j\geq 1}\prod_{y\in \bc_{\lambda}^{(j)}/\Conf^{\lambda} X}\left(b^{(j)}_{0,y} + b^{(j)}_{1,y} s + b^{(j)}_{2,y}  s^2 + \cdots \right)\\
& = & \sum_{\substack{(n_{i,j})_{i,j}\\ \sum_{i}n_{i,j} = n_j} } \left(\prod_{i,j} \Sym^{n_{i,j}}( b_{i}^{(j)}/\Conf^{\lambda} X ) \right)_{*} s^{\sum_{i,j} n_{i,j}i}
\end{eqnarray*}
in $\calM_{\Conf^{\lambda}(X)}$. Note that only terms satisfying $\sum_{i}n_{i,j} = n_j$ for every~$j$ will contribute since for each $j$, the above product over $\bc_{\lambda}^{(j)}$ relatively to $\Conf^{\lambda}X$ is finite, with~$n_j$ factors.
Using the expansion (\ref{eq.config-power-series-exp}) of $\Conf^{\lambda}_X(f(s))$, it remains to compare, for every collection of integers $(n_{i,j})_{i,j}$ such that $n_j = \sum_{i\geq 0} n_{i,j}$ for every $j\geq 1$, the classes of
$$\left( \prod_{i,j} \Sym^{n_{i,j}} (a_i)\right)_{*}\ \ \ \text{and}\ \ \   \left(\prod_{i,j} \Sym^{n_{i,j}}( b_{i}^{(j)}/\Conf^{\lambda} X ) \right)_{*} $$
in $\calM_{\Conf^{\lambda}(X)}$.
For this, observe that the projections $\bc_{\lambda}^{(j)} \to X$ induce the projection
$$\left({\prod_{i,j}} _{\ \ /\Conf^{\lambda} X} \Sym^{n_{i,j}} (\bc_{\lambda}^{(j)}/\Conf^{\lambda}X)\right)_{*} \to \left(\prod_{i,j} \Sym^{n_{i,j}}(X)\right)_{*},$$
(where the product on the left is taken relatively to $\Conf^{\lambda}(X)$), so that $$\left(\prod_{i,j} \Sym^{n_{i,j}}( b_{i}^{(j)}/\Conf^{\lambda} X ) \right)_{*}\in\calM_{\Conf^{\lambda} X}$$ will be the pullback of  $\left( \prod_{i,j} \Sym^{n_{i,j}} (a_i)\right)_{*}$ via this map.
On the other hand, this map is actually the identity: since for every $j$, we have $\sum_{i,j}n_{i,j} = j$, a point $(c_{i,j})_{i,j} \in \left(\prod_{i,j} \Sym^{n_{i,j}}(X)\right)_{*}$ completely determines the configuration above it.
\end{proof}

\subsection{Proof of Theorem \ref{theorem:conf-rv-body}}\label{subsect.conf-rv-proof}
 We proceed as in the proof of Theorem~\ref{theorem:convergence-criterion} to show that the limit can be expressed as the value of a certain series;  the results of \ref{subsect.power_series_conf_spaces} will then allow us to conclude.
 We write $n = \dim X$. By weak rationality,
 \begin{equation} \label{eq.first-conf-convergence}
(1-\LL^n t) \sum_{d\geq 0} [\Conf^d(X)]t^d = (1-\LL^nt)\prod_{x\in X}(1+t) = (1-\LL^nt) \frac{Z_X^{\Kap}(t)}{Z_X^{\Kap}(t^2)}
\end{equation}
 converges absolutely at $t = \LL^{-n}_{\phi}$ to an invertible element. In particular, the sequence of partial sums $[\Conf^{d}(X)]_\phi\LL_\phi^{-nd}$ converges to an invertible element as $d\to \infty$.

 On the other hand, denoting by $\bc_{\lambda}\to \Conf^{\lambda}(X)$ the universal configuration, note that the generating series of $\Conf^{\lambda\cdot *^d}(X)$  in $\calM_{C^{\lambda}(X)}$ has the following motivic Euler product decomposition
 \[ \sum_{d\geq 0}[ C^{\lambda \cdot *^d}(X)] t^d = \prod_{x \in \left(X \times C^\lambda X - \bc_{\lambda} \right) / C^\lambda X} (1+t). \]

 Consider the quotient of power series with coefficients in $\calM_{C^{\lambda}(X)}$
 \begin{align}\label{eq.conf-rv-power-series-quotient-relative} \frac{ \sum_{d\geq 0} [C^{\lambda \cdot *^d}(X)] t^d}{\sum_{d\geq 0} [\Conf^d(X) \times C^{\lambda}(X)] t^d} & = \frac{\prod_{x \in \left(X \times C^\lambda X - \bc_{\lambda} \right) / C^\lambda X} (1+t)}{\prod_{x \in \left(X \times C^\lambda X \right) / C^\lambda X}(1+t)}\\
 \nonumber & =  \prod_{x\in \bc_{\lambda}/\Conf^{\lambda}X}\frac{1}{1+t}.\end{align}
 Applying proposition \ref{prop:power-series-conf-space} and integrating over $C^\lambda(X)$, we obtain an identity of power series with coefficients in $\calM_K$ 
\begin{equation} \label{eq.conf-rv-power-series-quotient} \frac{ \sum_{d\geq 0} [C^{\lambda \cdot *^d}(X)] t^d}{\sum_{d\geq 0} [\Conf^d(X)] t^d} = C^\lambda_X\left(\frac{1}{1+t}\right). \end{equation}
By Lemma \ref{lemma:values_conf_series}, this power series converges absolutely at $t=\bbL_\phi^{-n}$ to $C^\lambda_X\left(\frac{1}{1+\bbL_\phi^{-n}}\right).$  

 Multiplying the left-hand side of (\ref{eq.conf-rv-power-series-quotient}) by $(1-\LL^n t) \sum_{d\geq 0} [\Conf^d(X)] t^d $, which as observed above also converges absolutely at $t=\bbL^{-n}_\phi$, we conclude that the series
 \[ (1-\LL^n t) \left(\sum_{d\geq 0}[C^{\lambda \cdot *^d}(X)] t^d\right) \]
  also converges absolutely at $t = \LL^{-n}_\phi$. In particular, the sequence of partial sums $[C^{\lambda \cdot *^d}(X) ]\LL^{-nd}$ converges, and we apply our usual trick to compute 
  \begin{align*} \lim_{d\rightarrow \infty}  \frac{[C^{\lambda\cdot*^d} (X)]_{\phi}}{[C^{|\lambda|+d}(X)]_{\phi}} & =  \bbL^{-n|\lambda|}_\phi\frac{\lim_{d\rightarrow \infty} [C^{\lambda \cdot*^d} (X)]_{\phi}\bbL_\phi^{-nd}}{\lim_{d\rightarrow \infty} [C^{d}(X)]_{\phi}\bbL_\phi^{-nd}} \\
 &= \bbL_\phi^{-n|\lambda|}\frac{ \left( (1-\bbL^n t)\sum_{d\geq 0} [C^{\lambda \cdot *^d}(X)] t^d \right)|_{t={\bbL_\phi^{-n}}}}{\left((1-\bbL^n t)\sum_{d\geq 0} [\Conf^d(X)]  t^d\right)|_{t=\bbL_\phi^{-n}}} \\
 &=\bbL_\phi^{-n|\lambda|} \left.\left(\frac{ \sum_{d\geq 0} [C^{\lambda \cdot *^d}(X)] t^d}{\sum_{d\geq 0} [\Conf^d(X)] t^d} \right)\right|_{t=\bbL_\phi^{-n}}\\
 &=\bbL_\phi^{-n|\lambda|}\left.\left( C^\lambda_X\left(\frac{1}{1+t}\right)\right)\right|_{t=\bbL_\phi^{-n}} \\
 &=\bbL_{\phi}^{-n|\lambda|}\Conf^{\lambda}_X\left(\frac{1}{1+\LL^{-n}_{\phi} }\right)\\
 &=  C^{\lambda}_{X} \left( \frac{1}{1 + \bbL^{n}_{\phi}}\right).
 \end{align*} 
 
\section{Hadamard convergence and cohomological stability}
\label{section:hadamard-convergence-cohomo-stab}
It is by now well-known (cf., e.g., \cite{ellenberg-venkatesh-westerland, church-ellenberg-farb:rep-stability-finite-fields, farb-wolfson:etale-stability})  that for a sequence of smooth varieties over $\bbF_q$, cohomological stability combined with suitable bounds on Betti numbers implies stabilization of point-counts through the Grothendieck-Lefschetz trace formula. By essentially the same computation, we show in Theorem~\ref{theorem:cohom-stab-hadamard} below that weight stabilization combined with suitable dimension bounds on the cohomology implies Hadamard stabilization.

Note that cohomological stablization implies weight stabilization (as long as the stabilization is as Galois representations; e.g., if the stabilization is realized by maps of algebraic varieties). In particular, point-counting results previously established via stable cohomology can be upgraded automatically to Hadamard stabilization: see Corollary~\ref{cor:coh_stab_hadamard} for a precise statement. For example, this gives an alternate proof of Theorem~\ref{theorem:labelled-configuration-spaces} for varieties admitting compactifiable lifts\footnote{Ho \cite{ho:free-factorization} has established \'e{t}ale homological stability for configuration spaces in positive characteristic without a lifting hypothesis, but for our result one would need the same for all colored configuration spaces as well.} to characteristic zero by applying the \'{e}tale representation stability and dimension bounds of Farb-Wolfson \cite{farb-wolfson:etale-stability}. Moreover, this also furnishes a natural strategy that may be useful in proving further cases of our meta-conjecture --- when weight convergence to a Hadamard function is known, to establish Hadamard convergence it will suffice to establish bounds on the Betti numbers.

\begin{theorem}\label{theorem:cohom-stab-hadamard}
Suppose $X_n/\bbF_q$ is a sequence of smooth varieties such that
\begin{enumerate}
\item There is a Hadamard function $Z_\infty(t)$ such that, in the weight topology,
\begin{equation}\label{eqn:had-cohom-stab} \lim_{n \rightarrow \infty} Z_{X_n}(t q^{-\dim X_n}) = Z_\infty(t). \end{equation}
\item There exist real numbers $C > 0$ and  $1\leq \lambda < \sqrt{q}$ such that, for any $n$, there exists a prime $\ell$ coprime to $q$ such that
\[ \dim_{\bbQ_\ell} H^i(X_{n, \overline{\bbF}_q}, \bbQ_{\ell}) \leq C \lambda^i \]
\end{enumerate}
Then (\ref{eqn:had-cohom-stab}) holds also in the Hadamard topology.
 \end{theorem}

\begin{remark}
We give the proof below, but first, some comments: 
\begin{enumerate}
\item The flexibility of allowing $\ell$ to vary with $n$ in (2) can be useful --- for example, this variation appears in the bounds for the Betti numbers of Hurwitz spaces established in \cite{ellenberg-venkatesh-westerland} (there the restriction $\ell > n$ arises because  at a point in the argument one needs the derived $S_n$-invariants in $\bbZ/\ell$ cohomology to be equal to the $S_n$-invariants).
\item The statement strikes a balance between brevity and utility, but the same method applies more generally: for example, for varieties that are not smooth with suitable bounds on compactly supported cohomology instead of cohomology, or to directly deduce the stabilization of $L$-functions in Remark~\ref{remark:l-function-stabilization} from cohomological stability and Betti bounds for the corresponding local systems as established (under lifting hypotheses) in \cite{farb-wolfson:etale-stability}.
\item  Ekedahl \cite{ekedahl:stacks} has defined a topology of polynomial growth refining the dimensional topology on the Grothendieck ring of varieties, and a slight modification of the proof of Theorem~\ref{theorem:cohom-stab-hadamard} shows that the zeta measure to the ring of Hadamard functions is continuous in this topology. The definition of Ekedahl's topology is a bit ad hoc, and one of our motivations for introducing the Hadamard topology was to find a more natural way to express a similar constraint.
\end{enumerate}
\end{remark}

\begin{proof}
We argue with divisors, i.e. elements of the completion of $\bbZ[\bbC^\times]$. So, write $D_n$ for the divisor attached\footnote{Recall from Section~\ref{subsect:witt_ring_rational} that to a meromorphic function $f$ we assign the divisor of $\frac{1}{f(1/t)}$, a normalization chosen so that the zeta function $\frac{1}{1-qt}$ of $\bbA^1_{\bbF_q}$ corresponds to $[q]$.} to $Z_{X_n}(t q^{-\dim X_n})$ and $D_\infty$ for the divisor attached to $Z_\infty(t)$.
For any divisor $D$, we write $\tau_m(D)$ for the part supported in the region $|z| \geq q^{-m/2}$. In particular, it is easy to see that
\[ \lim_{m \rightarrow \infty} \tau_m(D_\infty) = D_\infty \]
in the Hadamard topology.

On the other hand, by the definition of convergence in the weight topology, for any $m>0$, there exists $N>0$ such that for all $n \geq N$,
\begin{equation}\label{eqn:truncated-divisor} \tau_m(D_n)=\tau_m(D_\infty). \end{equation} Now, for any such $n$, taking $\ell$ as in (2) and fixing an embedding $\bbQ_{\ell} \rightarrow \bbC$, the Grothendieck-Lefschetz fixed point formula combined with Poincar\'{e} duality gives
\[ D_n = \sum_{i=0}^{2\dim X_n} (-1)^{i} [H^i(X_{n,\overline{\bbF}_q}, \bbQ_{\ell})^* \otimes_{\bbQ_{\ell}} \bbC].  \]
where the brackets denote taking the class in $K_0(\mathrm{Rep}_{\bbZ})$ (using the identification with $\bbZ[\bbC^{\times}]$ explained in Section~\ref{subsect:witt_ring_rational}). Then, combining (\ref{eqn:truncated-divisor}) with Deligne's \cite[Th\'{e}or\`{e}me I]{deligne:weilii} eigenvalue bounds which give that any eigenvalue $\alpha$ of Frobenius on $H^{i}(X_{n,\overline{\bbF}_q}, \bbQ_{\ell})^*$ satisfies $|\alpha|\leq q^{-i/2}$, we obtain
\[ \left|\left| D_n - \tau_m(D_\infty) \right|\right|_H \leq q^{-m/2}\sum_{i=0}^{m} \dim_{\bbQ_{\ell}} H^i(X_{n,\overline{\bbF}_q}, \bbQ_{\ell}) + \sum_{i=m+1}^{2\dim X_n} q^{-i/2} \dim_{\bbQ_{\ell}} H^i(X_{n,\overline{\bbF}_q}, \bbQ_{\ell}). \]
Invoking the bounds in hypothesis (2), we find this sum is bounded above by
\[ C\left(\frac{\lambda}{\sqrt{q}}\right)^{m} \left( m+1 + \frac{1}{1-\frac{\lambda}{\sqrt{q}}} \right), \]
and this bound goes to zero as $m \rightarrow \infty$ because $\lambda < \sqrt{q}.$ Thus, we conclude that also $D_n \rightarrow D_\infty$ in the Hadamard topology, as desired.
\end{proof}

In particular, we obtain the following corollary, making precise the statement that cohomological stabilization plus Betti bounds gives Hadamard convergence:
\begin{corollary}\label{cor:coh_stab_hadamard}Suppose $X_n / \bbF_q$ is a sequence of smooth varieties and $\ell \neq \mathrm{char}(\bbF_q)$ is a prime such that
\begin{enumerate}
\item For each $i \geq 0$, there is an $N \geq 0$ such that for all $n \geq N$ the $\Gal(\overline{\bbF}_q/\bbF_q)$-representations $H^i(X_{n,\overline{\bbF}_q}, \bbQ_\ell)^{\mathrm{ss}}$ are isomorphic to a fixed representation $H^i_\infty$ (here the superscript denotes semisimplification).
\item There exists $C > 0$ and  $\lambda < \sqrt{q}$ such that, for all $n$ and $i$
\[ \dim_{\bbQ_{\ell}} H^i(X_n, \bbQ_{\ell}) \leq C \lambda^i. \]

\end{enumerate}
Then, $Z_{X_n}(q^{-\dim X_n} t)$ converges in the Hadamard topology to
\[ \sum_{i=0}^{\infty}(-1)^i[H_{i,\infty}], \]
where here $H_{i,\infty}$ denotes the dual of $H^i_\infty$, with scalars extended to give a complex vector space by the choice of any embedding $\bbQ_{\ell} \rightarrow \bbC$.
\end{corollary}

\appendix

\section{Computations}\label{appendix:computations}

In \cref{table.limit-divisor} below we gives the first $250$ terms for the divisor 
\[ \lim_{(d_1,d_2) \rightarrow \infty} Z_{C^{\bullet^{d_1}\star^{d_2}}(\bbA^1_{\bbF_q})}\left(t q^{-(d_1+d_2)}\right), \]
where here the limit is in the weight topology. When $q=2$, this limiting divisor does not appear to correspond to a Hadamard function (cf. Remark~\ref{remark:hadamard-non-convergence}), but nonetheless our results show the sequence also converges in the point-counting topology. To further illustrate how this can occur, in \cref{table.exact-divisor} we give the exact formula of the divisor when $d_1=d_2=40$. One can then compute to see the cancellation for point-counting: the $q=2$ Hadamard norm is $395.538829916911$ but the $q=2$ point-counting semi-norm is $0.181319714263592$.

\newpage
\begin{center} 
\captionof{table}{First $250$ terms of $\lim_{(d_1,d_2) \rightarrow \infty} Z_{C^{\bullet^{d_1}\star^{d_2}}(\bbA^1_{\bbF_q})}\left(t q^{-(d_1+d_2)}\right)$.}
\label{table.limit-divisor}

{\footnotesize \ttfamily\begin{longtable*}{rr}
{\normalsize\normalfont $i$} & {\normalsize\normalfont coefficient of $[q^{-i}] $}\\ \hline
0 & 1 \\
1 & -3 \\
2 & 5 \\
3 & -10 \\
4 & 24 \\
5 & -55 \\
6 & 118 \\
7 & -250 \\
8 & 540 \\
9 & -1166 \\
10 & 2475 \\
11 & -5218 \\
12 & 11028 \\
13 & -23267 \\
14 & 48830 \\
15 & -102167 \\
16 & 213525 \\
17 & -445513 \\
18 & 927444 \\
19 & -1927166 \\
20 & 3999248 \\
21 & -8288404 \\
22 & 17153790 \\
23 & -35457313 \\
24 & 73212391 \\
25 & -151015163 \\
26 & 311189028 \\
27 & -640657585 \\
28 & 1317827566 \\
29 & -2708586539 \\
30 & 5562810556 \\
31 & -11416477207 \\
32 & 23413972647 \\
33 & -47988657094 \\
34 & 98296020099 \\
35 & -201224291653 \\
36 & 411703666030 \\
37 & -841899534112 \\
38 & 1720748369045 \\
39 & -3515328234048 \\
40 & 7178192714838 \\
41 & -14651215348621 \\
42 & 29891622362909 \\
43 & -60960729520648 \\
44 & 124274709833930 \\
45 & -253252619275830 \\
46 & 515905274269151 \\
47 & -1050598369362088 \\
48 & 2138748809597243 \\
49 & -4352556333294442 \\
50 & 8855142419299783 \\
51 & -18010175104285365 \\
52 & 36619803977908694 \\
53 & -74437884037740152 \\
54 & 151271098981190102 \\
55 & -307330496794545563 \\
56 & 624233017196670858 \\
57 & -1267601222149736713 \\
58 & 2573455649992469320 \\
59 & -5223384420459129280 \\
60 & 10599650504339968588 \\
61 & -21504939006993240476 \\
62 & 43620910664324846165 \\
63 & -88463413927558983487 \\
64 & 179369094441716234105 \\
65 & -363620936146947211139 \\
66 & 737003893336634408989 \\
67 & -1493525070407312843889 \\
68 & 3026071553583774207695 \\
69 & -6130160318373268357117 \\
70 & 12416305165787441941888 \\
71 & -25144482169769700475430 \\
72 & 50912518513077694908444 \\
73 & -103071775867327392158564 \\
74 & 208636308694120684565878 \\
75 & -422256725089154019803026 \\
76 & 854478921031350689093305 \\
77 & -1728883425555349537772147 \\
78 & 3497607554122553346247355 \\
79 & -7074876073480110138967125 \\
80 & 14309034136324437898603161 \\
81 & -28936554185525659439497151 \\
82 & 58509927574525580489119380 \\
83 & -118293195094794161305291884 \\
84 & 239132486275311262352998014 \\
85 & -483356022483243207472458684 \\
86 & 976891909453073575341693979 \\
87 & -1974139172168771694130837742 \\
88 & 3988980483500702718090433305 \\
89 & -8059348467160056857850301254 \\
90 & 16281439069254327401112771720 \\
91 & -32888298017019215826487386452 \\
92 & 66427309605631891852427285538 \\
93 & -134155799021691025521771284468 \\
94 & 270913472774442897719424024492 \\
95 & -547029744205693924330614993253 \\
96 & 1104463220953743112679968006155 \\
97 & -2229730248996000703271900417220 \\
98 & 4501060935168121797776308252779 \\
99 & -9085308660778996913683241591916 \\
100 & 18336963259621884186312937330536 \\
101 & -37006564130480842135420881234823 \\
102 & 74678296938835045366712451653048 \\
103 & -150686722006049417188882164883466 \\
104 & 304033289328605736801837757889052 \\
105 & -613385462035346120893462386389806 \\
106 & 1237407104790484830103451785458541 \\
107 & -2496083454393123235346360455263145 \\
108 & 5034699273149624209432668615272245 \\
109 & -10154451234898616906382452153763994 \\
110 & 20478984176727698229478103214975622 \\
111 & -41298085225883776492838324830483147 \\
112 & 83276312671638650634105812715706158 \\
113 & -167912729770850527045567532995866050 \\
114 & 338545314297377669135739970834099721 \\
115 & -682529641465357484326226939785869136 \\
116 & 1375935937942737803716090764040548543 \\
117 & -2773622488464218826498184276726243059 \\
118 & 5590740893108400697907757830485514668 \\
119 & -11268463168356409496415532461615555613 \\
120 & 22710868528096610367194049445753934421 \\
121 & -45769580932908800976831871086499330994 \\
122 & 92234769856874390273795578514079010526 \\
123 & -185860598612913044801841864730269995334 \\
124 & 374503005630066811126878392442709025314 \\
125 & -754569018607279234423667610817191802785 \\
126 & 1520262851219368816074332126572089460056 \\
127 & -3062772964719969636087115981981018953541 \\
128 & 6170035792263619351050455292566978783598 \\
129 & -12429042123060060151003262179674049196038 \\
130 & 25036007215716677420539355986395133276393 \\
131 & -50427824677622751805359126326377802814807 \\
132 & 101567199482982873264392954021879354129190 \\
133 & -204557360491172243478624563111868531696823 \\
134 & 411960370485300375547687520519502226821470 \\
135 & -829611500542651667331806049712787706316782 \\
136 & 1670603456362831206847362302632157611640072 \\
137 & -3363965724925806599011266967125981954243654 \\
138 & 6773444935084913876696775096583596298183108 \\
139 & -13637908675027441436057116768747781739434355 \\
140 & 27457838849409380315413796530515268407489675 \\
141 & -55279688909583876580585833350065956851424889 \\
142 & 111287337512971930313744554502136783509139787 \\
143 & -224030470620565185138114187204885173045566046 \\
144 & 450972288103407951076429498671472480900894500 \\
145 & -907766787667089046300297874897356031995072893 \\
146 & 1827177046255434145173073629090501998350189480 \\
147 & -3677639154926299735399831830273114930706933417 \\
148 & 7401844724023358326931783359492391415803949804 \\
149 & -14896814949606945772712282120059272199690249898 \\
150 & 29979866434245490844583144348449344360693631079 \\
151 & -60332177366298214408691600910066186887843803897 \\
152 & 121409187297844700570116626172426598599222626198 \\
153 & -244307928267034553210257870812929829060224442649 \\
154 & 491594743709525111581865915372346686734405187936 \\
155 & -989146826836016362658122074912175022534944715678 \\
156 & 1990207473240877597531413403273254338876802547040 \\
157 & -4004240887908066523314775718374722296383353278432 \\
158 & 8056130304128719786977903098713542750921357140249 \\
159 & -16207551036094612708384128936868306743927319581255 \\
160 & 32605669827995537882141376003470500998619660883838 \\
161 & -65592449277668135923145148203994025195285602114161 \\
162 & 131947066627455132690094280750865960630570254544336 \\
163 & -265418368252534732163388638198968661024046332532369 \\
164 & 533885006556978841673859509747827333034790892137060 \\
165 & -1073866157949160528842644873292715224609980727316442 \\
166 & 2159923824107976255787679933553822112189909230569860 \\
167 & -4344229119554769725531975946302444010351029324981531 \\
168 & 8737218126698544324983354343181783399234144113040996 \\
169 & -17571949987350926356148244810740733916329389383280610 \\
170 & 35338915495737920749569941345730118033124022425352917 \\
171 & -71067838406403965873181708846730412674466081241594550 \\
172 & 142915644912669381123934767287154566227048796839910300 \\
173 & -287391133596243553978025526428865096061369492892216231 \\
174 & 577901771904942496377897298680702776641561566472346509 \\
175 & -1162042191566025096073333560418006767663163414264761452 \\
176 & 2336560963156321258206594903332568731935777435716478404 \\
177 & -4698073670829204910429951149497168205623270962654074924 \\
178 & 9446048030016615435964025795014143220692909203774526167 \\
179 & -18991891895201248092411829357483627878037061328499916676 \\
180 & 38183364495333086204204682114930004931552019076194474230 \\
181 & -76765868701357294714795824370173618727343318391694633802 \\
182 & 154329973881712836323222680064105310721139741449888709502 \\
183 & -310256335734223324216076665641969459705588665497173883804 \\
184 & 623705279202696289295066521317341717885886961875906583159 \\
185 & -1253795440938300516089142300669579826263425588827900823773 \\
186 & 2520359987921702495659754659758416556389346309156935588716 \\
187 & -5066256883382087793564964993618615137251978971902397547338 \\
188 & 10183584999304906134008466464344962026597976993924855770002 \\
189 & -20469307350940835151747049199320044998319334089960769299637 \\
190 & 41142879284101383116184594018941378889033330476918655928999 \\
191 & -82694267690764386827363962624920969968162035653827651121126 \\
192 & 166205513942814486717606830876444989743453016535327886029097 \\
193 & -334044906428413611391355165977888136830209140798816585873928 \\
194 & 671357414335178326430744817824055982152630817133280117748150 \\
195 & -1349249723471416882316855584026818371620127437068003020111147 \\
196 & 2711568626261232162145679308550885878804628167867689653316398 \\
197 & -5449274402470410677251482438878923505195001441752145264632550 \\
198 & 10950820710870015616176840466748551469141706650750291515619432 \\
199 & -22006180491808074821226633271648780399307360717352367604021410 \\
200 & 44221429731125930866818700424117906695145598103780411815103936 \\
201 & -88860978351315253312957785221314608550508787788296296518972158 \\
202 & 178558157622312958198293390705731236882640973045841129901489278 \\
203 & -358788644063983246567535867359362571245002544503178214240054286 \\
204 & 720921801100698726386694533287526587431214885606660302982340624 \\
205 & -1448532341524246442342248986970439190664607865886748063901262327 \\
206 & 2910441593791677410571423695509416902856886245179841083766133211 \\
207 & -5847635883801003434025648107168165059200700685568114138941772899 \\
208 & 11748774930279755297246537565315174630066197777734430657895529207 \\
209 & -23604551767385068594821463748108605314319647758871879636797902301 \\
210 & 47423098592203576534355458696951983878051045433043063955216443887 \\
211 & -95274169946308147840780403360682157059441079565869026559277533215 \\
212 & 191404251025076306444836648744474710925748794337580471279140380135 \\
213 & -384520256153027720832343588144677015994231262916517617702578354310 \\
214 & 772463885412221247750959518172309275593339223655535204987540150838 \\
215 & -1551774249253673696738895239551544173613933026639442037615040708270 \\
216 & 3117240924567039756245256068732453168640627865443628864456247013851 \\
217 & -6261865649063678707968392455091112497430296740097472568689378905507 \\
218 & 12578496812172787574329459026104607521094908851061688080277086644833 \\
219 & -25266520517449908363513058769191519238709604487373171885812655365167 \\
220 & 50752086623453499599495814541713658375763803773380355410356385019761 \\
221 & -101942248171446376404811306396751761494336671084842843723533994465034 \\
222 & 204760613968597157212010029876713486804465127371302600563171124022586 \\
223 & -411273399298062141627923580121235278530544185511903002504259040064832 \\
224 & 826051014634815005139464364267018370249005956581169493760770952633261 \\
225 & -1659110210148733290252072210623434146652014922787564201813714175078037 \\
226 & 3332236283939183203228707149826950977481720527754055478448883897915665 \\
227 & -6692503307379732644393325590390792165758287355586962855209339873065199 \\
228 & 13441066134885157374941142465150996170874001629665169604682521868959572 \\
229 & -26994247425219599704022048659260610065138046610246664866230534368138664 \\
230 & 54212717456807333803446468226485466563565290356089792803625447444191968 \\
231 & -108873864845805016850142237474283157059409319666122162165797565224475771 \\
232 & 218644559248718136649200739632525468463253521897267370725297356832492893 \\
233 & -439082717498431167805788938314933663970201549385476217022929944640890819 \\
234 & 881752513762637047350955880772283529471535676582889977855830104064904050 \\
235 & -1770678948540417404656911704952446357822125847897764735776925045160337542 \\
236 & 3555705269942297117063252158576023481330861467174302311925764323753647955 \\
237 & -7140104354903296768142592107673112682036551683319793755880528208657622002 \\
238 & 14337594493516387018587442604981980152255607264327504228993603568963746649 \\
239 & -28789956891592793162580701973206192193451887385613153685521881169643775108 \\
240 & 57809442325457598502624005464830403280755001880407991296344942812757465951 \\
241 & -116077927318088147658885109110973474677707297047338797617322014115670291868 \\
242 & 233073911365640918583982238401098113091737014051890351051448047525110442208 \\
243 & -467983879434681066427015485809761438819926018767785270070916286071708841159 \\
244 & 939639759662761645826607434203471991500948183010290880113340609940625137998 \\
245 & -1886623297459580540003449408545708292094079321179393559084522176323337329589 \\
246 & 3787933707788042453135351202393917950224275817218928913807658304446716058235 \\
247 & -7605240761445472998617207473729266666313927167438658855852176879694593013499 \\
248 & 15269226468590914708207748063222774107557618777595720957055954826812698597855 \\
249 & -30655939363580606732561996894410171491023294598074730332881732843263810990419 \\
250 & 61546844703480411367930079662454293248852333175511986470297385611751909013867 \\
\end{longtable*}}

\bigskip

\newpage
\captionof{table}{The divisor $Z_{C^{\bullet^{d_1}\star^{d_2}}(\bbA^1_{\bbF_q})}\left(t q^{-(d_1+d_2)}\right)$ for $d_1 = d_2 = 40$.}
\label{table.exact-divisor}

{\footnotesize \ttfamily
\begin{longtable}{rrrrr}
{\normalsize\normalfont $i$} & {\normalsize\normalfont coefficient of $[q^{-i}]$} & & {\normalsize\normalfont $i$} & {\normalsize\normalfont coefficient of $[q^{-i}]$}  \\ \hline
0  & 1              &  & 40 & 7178192706102             \\
1  & -3             &  & 41 & -14651215147355           \\
2  & 5              &  & 42 & 29891619749371            \\
3  & -10            &  & 43 & -60960704596332           \\
4  & 24             &  & 44 & 124274516700328           \\
5  & -55            &  & 45 & -253251337471208          \\
6  & 118            &  & 46 & 515897749760655           \\
7  & -250           &  & 47 & -1050558448626228         \\
8  & 540            &  & 48 & 2138554410364751          \\
9  & -1166          &  & 49 & -4351677301167434         \\
10 & 2475           &  & 50 & 8851418068846937          \\
11 & -5218          &  & 51 & -17995282472068951        \\
12 & 11028          &  & 52 & 36563266171699586         \\
13 & -23267         &  & 53 & -74233098656280122        \\
14 & 48830          &  & 54 & 150560424516434836        \\
15 & -102167        &  & 55 & -304959028474877462       \\
16 & 213525         &  & 56 & 616599979165804930        \\
17 & -445513        &  & 57 & -1243838077427284749      \\
18 & 927444         &  & 58 & 2501726843328379367       \\
19 & -1927166       &  & 59 & -5013001590559463419      \\
20 & 3999248        &  & 60 & 9998887821446433255       \\
21 & -8288404       &  & 61 & -19831768033161995124     \\
22 & 17153790       &  & 62 & 39068695645092664153      \\
23 & -35457313      &  & 63 & -76346502518625937950     \\
24 & 73212391       &  & 64 & 147772610374307278073     \\
25 & -151015163     &  & 65 & -282802569405937507518    \\
26 & 311189028      &  & 66 & 533995597251319358859     \\
27 & -640657585     &  & 67 & -992196836341429077092    \\
28 & 1317827566     &  & 68 & 1807725664871875879551    \\
29 & -2708586539    &  & 69 & -3213777325101403051314   \\
30 & 5562810556     &  & 70 & 5535545342239809752109    \\
31 & -11416477207   &  & 71 & -9140023894242417884359   \\
32 & 23413972647    &  & 72 & 14237505104399687238881   \\
33 & -47988657094   &  & 73 & -20437644238323518538138  \\
34 & 98296020099    &  & 74 & 26158394219496320597829   \\
35 & -201224291653  &  & 75 & -28551600325000321382432  \\
36 & 411703666030   &  & 76 & 25039544129795914295127   \\
37 & -841899534112  &  & 77 & -16225470533207349260900  \\
38 & 1720748369045  &  & 78 & 6777326492252181076930    \\
39 & -3515328234048 &  & 79 & -1343840109164979124000

\end{longtable}}

\end{center}

\bibliography{references}
\bibliographystyle{plain}
\end{document}